\documentclass[12pt]{amsart}
\usepackage{amsmath,amssymb,amsfonts,amsthm}
\usepackage[all,cmtip]{xy}
\usepackage{fullpage}
\DeclareMathAlphabet{\mathpzc}{OT1}{pzc}{m}{it}
\usepackage[francais,english]{babel}
\begin{document}
\theoremstyle{plain}
\newtheorem{thm}{Theorem}[section]
\newtheorem{lem}[thm]{Lemma}
\newtheorem{cor}[thm]{Corollary}
\newtheorem{prop}[thm]{Proposition}
\newtheorem{rem}[thm]{Remark}
\newtheorem{defn}[thm]{Definition}
\newtheorem{ex}[thm]{Example}
\newtheorem{ques}[thm]{Question}
\newtheorem{fact}[thm]{Fact}
\newtheorem{conj}[thm]{Conjecture}
\numberwithin{equation}{section}
\def\theequation{\thesection.\arabic{equation}}
\newcommand{\mc}{\mathcal}
\newcommand{\mb}{\mathbb}
\newcommand{\surj}{\twoheadrightarrow}
\newcommand{\inj}{\hookrightarrow}
\newcommand{\zar}{{\rm zar}}
\newcommand{\an}{{\rm an}} 
\newcommand{\red}{{\rm red}}
\newcommand{\codim}{{\rm codim}}
\newcommand{\rank}{{\rm rank}}
\newcommand{\Ker}{{\rm Ker \,}}
\newcommand{\Pic}{{\rm Pic}}
\newcommand{\Div}{{\rm Div}}
\newcommand{\Hom}{{\rm Hom}}
\newcommand{\im}{{\rm im \,}}
\newcommand{\Spec}{{\rm Spec \,}}
\newcommand{\Sing}{{\rm Sing}}
\newcommand{\Char}{{\rm char}}
\newcommand{\Tr}{{\rm Tr}}
\newcommand{\Gal}{{\rm Gal}}
\newcommand{\Min}{{\rm Min \ }}
\newcommand{\Max}{{\rm Max \ }}
\newcommand{\CH}{{\rm CH}}
\newcommand{\pr}{{\rm pr}}
\newcommand{\cl}{{\rm cl}}
\newcommand{\gr}{{\rm Gr }}
\newcommand{\Coker}{{\rm Coker \,}}
\newcommand{\id}{{\rm id}}
\newcommand{\Rep}{{\bold {Rep} \,}}
\newcommand{\Aut}{{\rm Aut}}
\newcommand{\GL}{{\rm GL}}
\newcommand{\Bl}{{\rm Bl}}
\newcommand{\Jab}{{\rm Jab}}
\newcommand{\alb}{\rm Alb}
\newcommand{\NS}{\rm NS}
\newcommand{\sA}{{\mathcal A}}
\newcommand{\sB}{{\mathcal B}}
\newcommand{\sC}{{\mathcal C}}
\newcommand{\sD}{{\mathcal D}}
\newcommand{\sE}{{\mathcal E}}
\newcommand{\sF}{{\mathcal F}}
\newcommand{\sG}{{\mathcal G}}
\newcommand{\sH}{{\mathcal H}}
\newcommand{\sI}{{\mathcal I}}
\newcommand{\sJ}{{\mathcal J}}
\newcommand{\sK}{{\mathcal K}}
\newcommand{\sL}{{\mathcal L}}
\newcommand{\sM}{{\mathcal M}}
\newcommand{\sN}{{\mathcal N}}
\newcommand{\sO}{{\mathcal O}}
\newcommand{\sP}{{\mathcal P}}
\newcommand{\sQ}{{\mathcal Q}}
\newcommand{\sR}{{\mathcal R}}
\newcommand{\sS}{{\mathcal S}}
\newcommand{\sT}{{\mathcal T}}
\newcommand{\sU}{{\mathcal U}}
\newcommand{\sV}{{\mathcal V}}
\newcommand{\sW}{{\mathcal W}}
\newcommand{\sX}{{\mathcal X}}
\newcommand{\sY}{{\mathcal Y}}
\newcommand{\sZ}{{\mathcal Z}}
\newcommand{\A}{{\mathbb A}}
\newcommand{\B}{{\mathbb B}}
\newcommand{\C}{{\mathbb C}}
\newcommand{\D}{{\mathbb D}}
\newcommand{\E}{{\mathbb E}}
\newcommand{\F}{{\mathbb F}}
\newcommand{\G}{{\mathbb G}}
\renewcommand{\H}{{\mathbb H}}
\newcommand{\I}{{\mathbb I}}
\newcommand{\J}{{\mathbb J}}
\newcommand{\M}{{\mathbb M}}
\newcommand{\N}{{\mathbb N}}
\renewcommand{\P}{{\mathbb P}}
\newcommand{\Q}{{\mathbb Q}}
\newcommand{\R}{{\mathbb R}}
\newcommand{\T}{{\mathbb T}}
\newcommand{\V}{{\mathbb V}}
\newcommand{\W}{{\mathbb W}}
\newcommand{\X}{{\mathbb X}}
\newcommand{\Y}{{\mathbb Y}}
\newcommand{\Z}{{\mathbb Z}}
\newcommand{\Nwt}{{\rm Nwt}}
\newcommand{\Hdg}{{\rm Hdg}}
\newcommand{\ind}{{\rm ind \,}}
\newcommand{\Br}{{\rm Br}}
\newcommand{\inv}{{\rm inv}}
\newcommand{\Nm}{{\rm Nm}}
\newcommand{\Griff}{{\rm Griff}}
\newcommand{\Image}{\rm Im \,}
\newcommand{\Ev}{\rm Ev \,}
\title[Stable $\A^1$-homology sheaves]{Embedding of category of twisted Chow-Witt motives into geometric stable $\A^1$-derived category over a field}
\author{{Nguyen Le Dang Thi}}
\email{nguyen.le.dang.thi@gmail.com}
\date{28. 11. 2013}          
\begin{abstract}
We introduce in this note the notion of the category of twisted Chow-Witt correspondences $CHW(k)$ over a field $k$ of characteristic different from $2$, which leads to the category of twisted Chow-Witt motives $\widetilde{CHW}(k)$. Moreover, we show that over an infinite perfect field this category $\widetilde{CHW}(k)$ admits a fully faithful embedding into the geometric stable $\A^1$-derived category $D_{\A^1,gm}(k)$ after taking $\Q$-localization. We also prove a conjecture of F. Morel about the rational splitting of stable $\A^1$-cohomology over an essentially smooth scheme $S$ over a field $k$ of $char(k) \neq 2$.   
\end{abstract}
\maketitle
\tableofcontents
\section{Introduction}
One of the main motivations for this work is the embedding theorem of Voevodsky \cite{Voe00}, which asserts that there is a fully faithful embedding of the category of Grothendieck-Chow pure motives $\underline{Chow}(k)$ into the category of geometric motives $\bold{DM}_{gm}(k)$, hence also into the category of motives $\bold{DM}^-_{Nis}(k)$
$$\underline{Chow}(k)^{op} \rightarrow \bold{DM}_{gm}(k),$$ 
if $k$ is a perfect field, which admits resolution of singularities (see e.g. \cite[Prop. 20.1 and Rem. 20.2]{MVW06}, the assumption on resolution of singularities can be removed by using Poincar\'e duality). 
In this note, we construct a category $CHW(k)$, which we call the category of Chow-Witt correspondences over a field $k$ of characteristic different from $2$ and show that $\widetilde{CHW}(k)_{\Q}$ admits a fully faithful embedding into the geometric $\P^1$-stable $\A^1$-derived category $D_{\A^1,gm}(k)_{\Q}$ rationally, where $\widetilde{CHW}(k)$ is the category of Chow-Witt motives obtained from $CHW(k)$ by taking pseudo-abelian completion and formally inverting $\mathbb{L}$, the Lefschetz object. Our work can be viewed as an $\A^1$-version for Voevodsky's embedding theorem. The advantage here is that by using duality formalism for $\P^1$-stable $\A^1$-derived category $D_{\A^1}(k)$ established by A. Asok and C. Haesemeyer in \cite{AH11} (see \cite[App. A]{Hu05} for stable $\A^1$-homotopy categories), we do not have to assume the resolution of singularities. However, unlike in motivic setting, the main problem here is that we don't have cancellation theorem for the effective $\A^1$-derived category in general, see \cite[Rem. 3.2.4]{AH11}, that is the reason why we can prove the embedding result only for $\Q$-coefficient. F. Morel conjectured in general that: 
\begin{conj}\label{conjst}\cite{Mor04} 
Let $S$ be a regular Noetherian scheme of finite Krull dimension. One has a direct decomposition in the rationally motivic stable homotopy category $\mathbf{StHo}_{\A^1,\P^1}(S)$:  
$$[S^i,\G_m^{\wedge j}]_{\P^1}\otimes \Q = H^{j-i}_B(S,\Q(j)) \oplus H^{-i}_{Nis}(S,\underline{\bold{W}}\otimes \Q), $$
where $H^*_B(-,\Q(*))$ denotes the Beilinson motivic cohomology, $\underline{\bold{W}}$ is the unramified Witt sheaf and $[-,-]_{\P^1} \otimes \Q$ denotes $\Hom_{\mathbf{StHo}_{\A^1,\P^1}(S)}(-,-)_{\Q}$. 
\end{conj}
In this note we will prove this conjecture (see Thm. \ref{proofofconjst}) in case $ S \rightarrow \Spec (k)$ is an essentially smooth $k$-scheme, where $k$ is a field of $char(k) \neq 2$. This result was previously announced by F. Morel, but the proof was never published. In fact, the proof of this conjecture is one of the main step to prove our embedding result. However, the most challenging part is to show $\mathbf{1}_{\Q+} = H_B$, which has been done in \cite[\S 16.2]{CD10}. Our work is to identify $\mathbf{1}_{\Q-} = H\mathbf{W}_\Q$. The proof follows the same strategy as in \cite{CD10}. On the other hand, our interest started originally from the study of the existence of $0$-cycles of degree one on algebraic varieties. More precisely, H\'el\`ene Esnault asked (cf. \cite{Lev10}): Given a smooth projective variety $X$ over a field $k$, such that $X$ has a zero cycle of degree one. Are there "motivic" explanations which give the (non)-existence of a $k$-rational point? In \cite{AH11}, A. Asok and C. Haesemeyer show that the existence of zero cycles of degree one over an infinite perfect field of $char(k) \neq 2$ is equivalent to the assertion that the structure map $\bold{H}^{st\A^1}_0(X) \rightarrow \bold{H}^{st \A^1}_0(\Spec k)$ is a split epimorphism, where $\bold{H}^{st\A^1}_i(X)$ denotes the $\P^1$-stable $\A^1$-homology sheaves, while in an earlier work \cite{AH11a} they also showed that the existence of a $k$-rational point over an arbitrary field $k$ is equivalent to the condition that the structure map $\bold{H}_0^{\A^1}(X) \rightarrow \bold{H}_0^{\A^1}(\Spec k)$ is split surjective. So roughly speaking, the obstruction to the lifting of a zero cycle of degree one to a rational point arises by passing from $S^1$-spectra to $\P^1$-spectra. As remarked by M. Levine, it is not to expect that the category of Chow-Witt correspondences $CHW(k)$ contains any information about the existence of rational points. Now we state our main theorem in this work:               
\begin{thm}\label{mainthm}
Let $k$ be a field of $char(k) \neq 2$. There exists a category $CHW(k)$, whose objects are such a pair $(X,\omega_{X/k})$, where $X \in SmProj/k$ and $\omega_{X/k}$ denotes its canonical line bundle. The  morphisms in $CHW(k)$ are given by 
$$\Hom_{CHW(k)}(X,Y) = \widetilde{\CH}^{\dim(X)}(X\times Y, p^{{XY}*}_X\omega_{X/k}).$$  
If $k$ is an infinite perfect field of $char(k) \neq 2$, then one has a fully faithful embedding 
$$\widetilde{CHW}(k)_{\Q}^{op} \rightarrow D_{\A^1,gm}(k)_{\Q}.$$
\end{thm} 
In fact, one of the main steps in the work of \cite{AH11} is to exhibit a natural isomorphism $H_0^{st\A^1}(X)(L) \rightarrow \widetilde{\CH}_0(X_L)$ for any separable, finitely generated field extension $L/k$. So one may relate this step to our work as evaluating at a generic point, but much weaker than expected, since we can only prove the result for $\Q$-coefficient. Now our paper is organized as follows: we will review shortly $\A^1$-homotopy theory in section \S 2. Section \S 3 is devoted for $\A^1$-derived categories, in fact we will define the geometric $\P^1$-stable $\A^1$-derived category $D_{\A^1,gm}(k)$ over a field $k$ in \ref{defgmA1} at the end of \S 3. In these \S 2 and \S 3 we simply steal everything which is needed from the presentation of \cite{AH11}. For a complete treatment we strongly recommend the reader to \cite{Ay08}, \cite{CD10} and \cite{Mor12}. In section \S 4 we prove the Morel's conjecture \ref{conjst} and also give some applications, which are again unpublished results of Morel. More precisely, based on the slice spectral sequence of Voevodsky and the recent work of R\"ondigs and \O stv\ae r \cite{RO13} computing the slice of Hermitian $K$-theory and on the fact that rational motivic cohomology $H\Q$ is the motivic Landweber exact spectrum associated with the additive formal group law over $\Q$ of Naumann, Spitzweck and \O stv\ae r \cite{NSO09}, we obtain a splitting 
$$\mathbf{KO}_{\Q+} = \bigvee_{n \in \Z} \P^{1 \wedge 2n}_{\Q+},$$
in $SH(k)_\Q$, where $k$ is a perfect field of $char(k) \neq 2$ (see corollary \ref{corsssKO}). An interesting consequence of the theorem \ref{proofofconjst} is that we also have a splitting (see \ref{corsplitKT}) 
$$\mathbf{KO}_{\Q-} = \bigvee_{n \in \Z} S^{4n}_{s,\Q-}.$$
These together will imply the calculation for Hermitian K-theory of a smooth scheme over a perfect field $k$ of $char(k) \neq 2$:
$$KO^{p,q}(X)_\Q = \bigoplus_{n \in \Z} H^{p+4n,q+2n}_M(X,\Q) \bigoplus_{m \in \Z}H^{p+4m}_{Nis}(X,\underline{\mathbf{W}}_\Q).$$
Another result we also obtain from the theorem \ref{proofofconjst} is what we called motivic Serre's theorem, which roughly tells us that over a perfect field of $char(k) \neq 2$ all the rational homotopy groups of the motivic sphere spectrum are trivial (see corollary \ref{cormotivicserre}). The strategy of the proof of theorem \ref{proofofconjst} follows almost in the same line as the proof for the $\Q_+$-part in \cite[Part 4, \S 16]{CD10}. The proof of Theorem \ref{mainthm} is divided in two steps. Firstly, we construct $CHW(k)$ in section \S 5  (see Proposition \ref{propCW} and Def. \ref{defCHW}). The embedding functor is established then in section \S 6. Since we work with $\Q$-coefficient, we also have to point out that even our motivation comes from the existence of zero cycles of degree one of an algebraic varieties over a field, but we can not expect that our result can contain any precise information about this. We fix now some notations throughout this work. For a pair of adjoint functors $F: \sA \rightarrow \sB$ and $G: \sB \rightarrow \sA$, we will adopt the notation in \cite{CD10} 
$$F: \sA \rightleftarrows \sB : G,$$
where $F$ is left adjoint to $G$ and $G$ is right adjoint to $F$. The unit $1 \rightarrow GF$ will be denoted by $ad(F,G)$ and the counit $FG \rightarrow 1$ by $ad'(F,G)$. Given two smooth $k$-schemes $X,Y \in Sm/k$ and two vector bundles $\sE, \sE'$ over $X$ resp. $Y$, we write $\sE \times \sE'/ X\times Y$ for the external sum over $X \times_k Y$, or sometime just $\sE \boxplus \sE'$. The $\P^1$- stable homotopy category over a base scheme $S$ will be denoted by $\mathbf{StHo}_{\A^1,\P^1}(S)$ and we write $\mathbf{StHo}_{\A^1,S^1}(S)$ for the $S^1$-stable homotopy category. Sometime, when it is clear about which category we are talking about, we just abbreviate our $\P^1$-stable homotopy category by $SH(S)$.    
\section{$\A^1$-homotopy category}
\subsection{Unstable $\A^1$-homotopy category}
Let $Sm/k$ denote the category of separated smooth schemes of finite type over a field $k$. We write $Spc/k$ for the category $\Delta^{op}Sh_{Nis}(Sm/k)$ consisting of simplicial Nisnevich sheaves of sets on $Sm/k$. An object in $Spc/k$ is simply called a $k$-space, which is usually denoted by calligraphic letter $\sX$. The Yoneda embedding $Sm/k \rightarrow Spc/k$
is given by sending a smooth scheme $X \in Sm/k$ to the corresponding representable sheaf $\Hom_{Sm/k}(-,X)$ then by taking the associated constant simplicial object, where all face and degeneracy maps are the identity. We will identify $Sm/k$ with its essential image in $Spc/k$. Denote by $Spc_+/k$ the category of pointed $k$-space, whose objects are $(\sX,x)$, where $\sX$ is a $k$-space and $x: \Spec k \rightarrow \sX$ is a distinguished point. One has an adjoint pair 
$$ Spc/k \rightleftarrows Spc_+/k, $$
which means that the functor $Spc/k \rightarrow Spc_+/k$ sending $\sX \rightarrow \sX_+ = \sX \coprod \Spec k$ is left-adjoint to the forgetful functor $Spc_+/k \rightarrow Spc/k$. The category $Spc/k$ can be equipped with the injective local model structure $(C_s,W_s,F_s)$, where cofibrations are monomorphisms, weak equivalences are stalkwise weak equivalences of simplicial sets and fibrations are morphisms with right lifting property wrt. morphisms in $C_s \cap W_s$. Denote by $\bold{Ho}_s^{Nis}(k)$ the resulting unpointed homotopy category as constructed by Joyal-Jardine (cf. \cite[\S 2 Thm.  1.4]{MV01}). We will write $\bold{Ho}_{s,+}^{Nis}(k)$ for the pointed homotopy category. 
\begin{defn}\cite{MV01}\label{def2.1}
\begin{enumerate}
\item A $k$-space $\sZ \in Spc/k$ is called $\A^1$-local if and only for any object $\sX \in Spc/k$, the projection $\sX \times \A^1 \rightarrow \sX$ induces a bijection 
$$\Hom_{\bold{Ho}_s^{Nis}(k)}(\sX,\sZ) \stackrel{\simeq}{\rightarrow} \Hom_{\bold{Ho}_s^{Nis}(k)}(\sX \times \A^1,\sZ). $$
\item Let $\sX \rightarrow \sY \in Mor(Spc/k)$ be a morphism of $k$-spaces. It is an $\A^1$-weak equivalence if and only for any $\A^1$-local object $\sZ$, the induced map 
$$\Hom_{\bold{Ho}_s^{Nis}(k)}(\sY,\sZ) \rightarrow \Hom_{\bold{Ho}_s^{Nis}(k)}(\sX,\sZ) $$ 
is bijective.  
\end{enumerate}
\end{defn}       
In \cite[\S 2 Thm. 3.2]{MV01}, F. Morel and V. Voevodsky proved that $Spc/k$ can be endowed with the $\A^1$-local injective model structure $(C,W_{\A^1},F_{\A^1})$, where cofibrations are monomorphisms, weak equivalences are $\A^1$-weak equivalences. The associated homotopy category obtained from $Spc/k$ by inverting $\A^1$-weak equivalences is denoted by $\bold{Ho}_{\A^1}(k) \stackrel{def}{=} Spc/k[W_{\A^1}^{-1}] $. This category is called the unstable $\A^1$-homotopy category of smooth $k$-schemes. Let $\bold{Ho}_{s,\A^1-loc}^{Nis}(k) \subset \bold{Ho}_s^{Nis}(k)$ be the full subcategory consisting of $\A^1$-local objects. In fact, one has an adjoint pair (cf. \cite{MV01})
$$L_{\A^1}: \bold{Ho}_s^{Nis}(k) \rightleftarrows \bold{Ho}_{s,\A^1-loc}^{Nis}(k) : i, $$
where $L_{\A^1}$ is the $\A^1$-localization functor sending $\A^1$-weak equivalences to isomorphisms. $L_{\A^1}$ induces thus an equivalence of categories $\bold{Ho}_{\A^1}(k) \rightarrow \bold{Ho}_{s,\A^1-loc}^{Nis}(k)$. This will imply that if $\sX \in Spc/k$ is any object and $\sY$ is an $\A^1$-local object, then one has a canonical bijection 
$$\Hom_{\bold{Ho}_s^{Nis}(k)}(\sX,\sY) \stackrel{\simeq}{\rightarrow} \Hom_{\bold{Ho}_{\A^1}(k)}(\sX,\sY). $$  
We will write $\bold{Ho}_{\A^1,+}(k)$ for the unstable pointed $\A^1$-homotopy category of smooth $k$-schemes. Recall 
\begin{defn}\label{defThomsp} 
Let $X \in Sm/k$ and $E$ be a vector bundle over $X$. The Thom space of $E$ is the pointed sheaf 
$$Th(E/X) = E/E-s_0(X), $$
where $s_0: X \rightarrow E$ is the zero section of $E$.
\end{defn}
Let $T \in Spc_+/k$ be the quotient sheaf $\A^1/(\A^1- \{0\})$ pointed by the image of $\A^1-\{0\}$. Then $T \cong S^1_t \wedge S^1_s$ in $\bold{Ho}_{\A^1,+}(k)$ (\cite[Lem. 2. 15]{MV01}). For a pointed space $\sX \in Spc_+/k$, we denote by $\Sigma_T(\sX,x) = T \wedge (\sX,x)$. Remark that $\P^n/\P^{n-1} \cong T^n \stackrel{def}{=} T^{\wedge n}$ is an $\A^1$-equivalence. In particular, we have $(\P^1,*) \cong T$ (\cite[Cor. 2.18]{MV01}). Recall 
\begin{prop}\cite[\S 3 Prop. 2. 17]{MV01}\label{propthom}
Let $X, Y \in Sm/k$ and $E, E'$ be vector bundles on $X$ and $Y$ respectively. One has 
\begin{enumerate} 
\item There is a canonical isomorphism of pointed sheaves 
$$Th(E \times E' / X \times Y) = Th(E/X) \wedge Th(E'/Y). $$
\item There is a canonical isomorphism of pointed sheaves 
$$Th(\sO_X^n) = \Sigma^n_T X_+$$
\item The canonical morphism of pointed sheaves 
$$\P(E \oplus  \sO_X)/\P(E) \rightarrow Th(E)$$ 
is an $\A^1$-weak equivalence. 
\end{enumerate}
\end{prop}
The following theorem due to Voevodsky will play an essential role for our purpose. 
\begin{thm}\cite[Thm. 2.11]{Voe03}\label{ThmVoeThom}
Let $X \in SmProj/k$ a smooth projective variety of pure dimension $d_X$ over a field $k$. There exists an integer $n_X$ and a vector bundle $V_X$ over $X$ of rank $n_X$, such that  
$$ V_X \oplus T_X = \sO_X^{n_X+d_X} \in K_0(X),$$
where $T_X$ denotes the tangent bundle of $X$. Moreover, there exists a morphism $T^{\wedge n_X + d_X} \rightarrow Th(V_X)$ in $\bold{Ho}_{\A^1,+}(k)$, such that the induced map $H^{2d_X}_{\sM}(X,\Z(d_X)) \rightarrow \Z$ coincides with the degree map $\deg: \CH_0(X) \rightarrow \Z$, where $T = S^1_s \wedge \G_m$.
\end{thm}
\subsection{Stable $\A^1$-homotopy category} 
Let $\bold{Spect}^{\Sigma}(Spc/k)$ be the category of symmetric spectra in $k$-spaces, which can be viewed as category of Nisnevich sheaves of symmetric spectra. By applying the construction in \cite[Def. 4.4.40, Cor. 4.4.42, Prop. 4.4.62]{Ay08}, $\bold{Spect}^{\Sigma}(Spc/k)$ has the structure of a monoidal model category. Let $\bold{StHo}_{S^1}(k)$ be the resulting homotopy category. The stable $\A^1$-homotopy category of $S^1$-spectra $\bold{StHo}_{\A^1,S^1}(k)$ is obtained from $\bold{StHo}_{S^1}(k)$ by Bousfield localization. Equivalently, the category $\bold{Spect}^{\Sigma}(Spc/k)$ can be equipped with an $\A^1$-local model structure (cf. \cite[Def. 4.5.12]{Ay08}). The homotopy category of this $\A^1$-local model structure is $\bold{StHo}_{\A^1,S^1}(k)$, which is also known to be equivalent to the category $\bold{StHo}^{S^1}_{\A^1-loc}(Sm/k)$ constructed by F. Morel in \cite[Def. 4.1.1]{Mor05}. The $\A^1$-local symmetric sphere spectrum is defined by taking the functor 
$$\underline{n} \mapsto L_{\A^1}(S_s^{1 \wedge n}) $$
with an action of symmetric groups, where $L_{\A^1}$ denotes the $\A^1$-localization functor. For a pointed space $(\sX,x)$, its $\A^1$-local symmetric suspension spectrum is defined as the symmetric sequence 
$$\underline{n} \mapsto L_{\A^1}(S_s^{1 \wedge n} \wedge \sX) $$
together with symmetric groups actions. Let $\sE$ be an $\A^1$-local symmetric
spectrum in $Spc/k$. One defines (\cite[Def. 2.1.11]{AH11}) the $i$-th
$S^1$-stable $\A^1$-homotopy sheaf $\pi_i^{st\A^1,S^1}(\sE)$ of $\sE$ as the
Nisnevich sheaf on $Sm/k$ associated to the presheaf 
$$U \mapsto \Hom_{\bold{StHo}_{\A^1,S^1}(k)}(S^{1 \wedge i}_s \wedge \Sigma^{\infty}_s U_+,\sE). $$
Now we consider the symmetric $T$-spectra or $\P^1$-spectra. $\P^1$ is pointed
with $\infty$ and $\P^{1 \wedge n}$ has a natural action of $\Sigma_n$ by
permutation of the factors, so the association $\underline{n} \mapsto \P^{1
\wedge n}$ is a symmetric sequence. A symmetric $\P^1$-spectrum is a symmetric
sequence with a module structure over the sphere spectrum $\mathbf{S}^0$. Denote
by $\bold{Spect}_{\P^1}^{\Sigma}(Spc/k)$ the full subcategory of the category of
symmetric sequence in $k$-spaces $\bold{Fun}(\mathpzc{Sym},Spc_+/k)$
consisting of
symmetric $\P^1$-spectra, which also has a model structure \cite[Def.
4.5.21]{Ay08}. Here we denote by $\mathpzc{Sym}$ the groupoid, whose objects
are $\underline{n}$ and morphisms are given by bijections. Let
$\bold{StHo}_{\A^1,\P^1}(k)$ be the resulting homotopy
category, which is called $\P^1$-stable $\A^1$-homotopy category. For a pointed
space $(\sX,x)$, we will write $\Sigma^{\infty}_{\P^1}(\sX,x)$ for the
suspension symmetric $\P^1$-spectrum, i.e., it is given by the functor
$\underline{n} \mapsto \P^{1 \wedge n} \wedge \sX$ equipped with an action of
symmetric group by permuting the first $n$-factors. Let $\bold{S}^i$ be a
suspension symmetric $\P^1$-spectrum of $S^i_s$. If $\sE$ is a symmetric
$\P^1$-spectrum, then the $i$-th $\P^1$-stable $\A^1$-homotopy sheaf
$\pi_i^{st\A^1,\P^1}(\sE)$ is defined as the Nisnevich sheaf on $Sm/k$
associated to the presheaf (cf. \cite[Def. 2.1.14]{AH11}) 
$$U \mapsto \Hom_{\bold{StHo}_{\A^1,\P^1}(k)}(\bold{S}^i \wedge
\Sigma^{\infty}_{\P^1}U_+,\sE).$$ 
\begin{thm}\label{thmMorS1}\cite[Thm. 6.1.8 and Cor. 6.2.9]{Mor05}
 Let $\sE$ be an $\A^1$-local symmetric $S^1$-spectrum. The homotopy
sheaves $\pi_i^{st\A^1,S^1}(\sE)$ are strictly $\A^1$-invariant.    
\end{thm}
One has a canonical isomorphism \cite[Prop. 2.1.16]{AH11}
\begin{multline*}
colim_n\Hom_{\bold{StHo}_{\A^1,S^1}(k)}(\Sigma^{\infty}_s \G_m^{\wedge n}
\wedge \Sigma^{\infty}_s(U_+), \Sigma^{\infty}_s \G_m^{\wedge n} \wedge
\Sigma^{\infty}_s(\sX,x)) \stackrel{\cong}{\rightarrow} \\
\Hom_{\bold{StHo}_{\A^1,\P^1}(k)}(\Sigma^{\infty}_{\P^1}(U_+),\Sigma^{\infty}
_{\P^1}(\sX,x)).
\end{multline*}
So one may view that $\bold{StHo}_{\A^1,\P^1}(k)$ is obtained from
$\bold{StHo}_{\A^1,S^1}(k)$ by formally inverting the $\A^1$-localized
suspension spectrum of $\G_m$. So from \ref{thmMorS1}, we see that for a
pointed $k$-space $(\sX,x)$, the homotopy sheaves $\pi_i^{st\A^1,\P^1}(\sX)$
are also strictly $\A^1$-invariant. By the computation of F. Morel, one can
identify the Milnor-Witt $K$-theory sheaves with stable homotopy sheaves of
spheres 
$$\bold{K}^{MW}_n \stackrel{def}{=}
\pi_0^{st\A^1,\P^1}(\Sigma^{\infty}_{\P^1}(\G_m^{\wedge n})).$$
This identification allows us to conclude that $\bold{K}^{MW}_n$ are strictly
$\A^1$-invariant sheaves.  
\section{$\A^1$-homological algebra}
\subsection{Effective $\A^1$-derived category}
Let $Ch_{-}(\sA b_k)$ be the category of chain complexes over the category $\sA
b_k$ of abelian Nisnevich sheaves. Denote by $Ch_{\geq 0}(\sA b_k)$ the
category of chain complexes of abelian Nisnevich sheaves, whose homoglocial
degree $\geq 0$. The sheaf-theoretical Dold-Kan correspondence 
$$N: \Delta^{op}\sA b_k \rightleftarrows Ch_{\geq 0}(\sA b_k) : K,$$
where $\Delta^{op} \sA b_k$ is the cateogry of simplicial abelian Nisnevich
sheaves, gives us via the inclusion functor $Ch_{\geq 0}(\sA b_k) \inj
Ch_{-}(\sA b_k)$, a functor 
$$\Delta^{op}(\sA b_k) \rightarrow Ch_{-}(\sA b_k).$$
By applying this functor on the Eilenberg-Maclane spectrum $H\Z$, we obtain a
ring spectrum $\widetilde{H\Z}$ in $\bold{Fun}(\mathpzc{Sym},Ch_{-}(\sA b_k))$.
Let $\bold{Spect}^{\Sigma}(Ch_{-}(\sA b_k))$ be the full subcategory of the
category $\bold{Fun}(\mathpzc{Sym},Ch_{-}(\sA b_k))$ consisting of modules over
$\widetilde{H\Z}$. On the other hand, by composing with the free abelian group
functor 
$$\Z (-): Spc/k \rightarrow \Delta^{op}(\sA b_k),$$
one obtains a functor 
$$\bold{Fun}(\mathpzc{Sym},Spc_+/k) \rightarrow
\bold{Fun}(\mathpzc{Sym},Ch_{-}(\sA b_k)),$$
which sends the sphere symmetric sequence to $\widetilde{H\Z}$. This induces
then a functor between categories of symmetric spectra 
$$\bold{Spect}^{\Sigma}(Spc/k) \rightarrow \bold{Spect}^{\Sigma}(Ch_{-}(\sA
b_k)).$$ 
In fact, by \cite[Thm. 9.3]{Hov01}, this induces a Quillen functor, which one refers as Hurewicz functor 
$$\mathfrak{H}^{ab}: \bold{StHo}_{S^1}(k) \rightarrow D_{-}(\sA b_k). $$
Now the effective $\A^1$-derived category $D^{eff}_{\A^1}(k)$ is constructed by applying $\A^1$-localization on the category $\bold{Spect}^{\Sigma}(Ch_{-}(\sA b_k))$. By the work of Cisinski and D\'eglise (cf. \cite[\S 5]{CD10}), this category is equivalent to the $\A^1$-derived category constructed by F. Morel in \cite{Mor12}. Let $(\sX,x) \in Spc_+/k$ be a pointed space, and $\Sigma^{\infty}_s(\sX,x)$ its suspension symmetric spectrum. We apply the Hurewicz functor on $\Sigma^{\infty}_s(\sX,x)$ and then $L_{\A^1}^{ab}(-)$, so we may define a functor 
$$\widetilde{C}_*^{\A^1}: \bold{StHo}_{S^1}(k) \rightarrow D^{eff}_{\A^1}(k), \quad \Sigma^{\infty}_s(\sX,x) \mapsto L_{\A^1}^{ab}(\mathfrak{H}^{ab}(\Sigma^{\infty}_s(\sX,x))). $$
Here we write $L_{\A^1}^{ab}$ for the $\A^1$-localization functor on chain complexes to distinguish from the $\A^1$-localization $L_{\A^1}$ on spaces.  If $\sX \in Spc/k$ is not pointed, then we write $C^{\A^1}_*(\sX) \stackrel{def}{=} \widetilde{C}^{\A^1}_*(\sX_+)$. Define $\Z[n] = \mathfrak{H}^{ab}(\Sigma^{\infty}_s S^{n}_s)$. 
\begin{defn}\label{defA1homology}
Let $\sX \in Spc/k$ be a $k$-space. Its $i$-th $\A^1$-homology sheaf is the Nisnevich sheaf $\bold{H}^{\A^1}_i(\sX)$ associated to the presheaf 
$$U \mapsto \Hom_{D^{eff}_{\A^1}(k)}(C_*^{\A^1}(U)[i],C_*^{\A^1}(\sX)) \stackrel{def}{=} \Hom_{D^{eff}_{\A^1}(k)}(C_*^{\A^1}(U) \otimes \Z[i],C_*^{\A^1}(\sX)). $$ 
\end{defn} 
Consider $(\P^1,\infty)$ pointed by $\infty$. According to \cite[Cor. 2.18]{MV01}, we have $\P^1 = S^1_s \wedge \G_m$, so we have an identification $\widetilde{C}_*^{\A^1}(\P^1) = \widetilde{C}^{\A^1}_*(S^1_s \wedge \G_m)$. We define the $\A^1$-Tate complex (called enhanced Tate (motivic) complex by A. Asok and C. Haesemeyer \cite[Def. 2.1.25 and Def. 3.2.1 and Lem. 3.2.2]{AH11}) as 
$$\Z_{\A^1}(n) \stackrel{def}{=} \widetilde{C}_*^{\A^1}(\P^{1 \wedge n})[-2n] = \Z_{\A^1}(1)^{\otimes n}. $$
\begin{defn}\label{defA1unstcoh}
Let $\sX \in Spc/k$ be a $k$-space. The bigraded unstable $\A^1$-cohomology group $H^{p,q}_{\A^1}(\sX,\Z)$ is defined as 
$$H^{p,q}_{\A^1}(\sX,\Z) = \Hom_{D^{eff}_{\A^1}(k)}(C_*^{\A^1}(\sX),\Z_{\A^1}(q)[p]). $$
\end{defn} 
The relationship between unstable $\A^1$-cohomology and Nisnevich hypercohomology with coefficient $\Z_{\A^1}(n)$ is given by the following 
\begin{prop}\label{propunstcoh}\cite[Prop. 3.2.5]{AH11} 
Let $k$ be a field and $\sX \in Spc/k$ be a $k$-space. One has 
\begin{enumerate}
\item For any $p,q$, there is a canonical isomorphism 
$$\H^p_{Nis}(\sX,\Z_{\A^1}(q)) \stackrel{\simeq}{\rightarrow} H^{p,q}_{\A^1}(\sX,\Z). $$
\item The cohomology sheaves $\underline{H}^p(\Z_{\A^1}(q)) = 0$, if $p > q$. 
\item There is a canonical isomorphism $\underline{H}^p(\Z_{\A^1}(p)) \cong \bold{K}^{MW}_p$, for all $p >0$. 
\end{enumerate}
\end{prop} 
\begin{rem}\label{remHI}{\rm
One observes that the complex $\Z_{\A^1}(n)$ is $\A^1$-local, hence by definition (cf. \cite[Def. 5.17]{Mor12}) one has immediately that the sheaves $\underline{H}^p(\Z_{\A^1}(q))$ are strictly $\A^1$-invariant. 
}
\end{rem}
\subsection{$\P^1$-stable $\A^1$-derived category}
Having $\A^1$-Tate complex, the way that we stabilize the category $D^{eff}_{\A^1}(k)$ is to invert formally the $\A^1$-Tate complex to obtain the $\P^1$-stable $\A^1$-derived category $D_{\A^1}(k)$. This can be done by \cite[\S 5]{CD10}. As before, we take $D_{\A^1}(k)$ as the resulting homotopy category of the model category $\bold{Spect}^{\Sigma}_{\P^1}(Ch_{-}(\sA b_k))$ consisting of modules over the $\A^1$-localization of the normalized chain complex of the free abelian group on the sphere symmetric $\P^1$-spectrum. For a pointed space $(\sX,x) \in Spc_+/k$, the stable $\A^1$-complex $\widetilde{C}_*^{st\A^1}(\sX)$ of $(\sX,x)$ is defined as $L_{\A^1}^{ab}(N\Z(\Sigma^{\infty}_{\P^1}(\sX,x)))$ and if $\sX \in Spc/k$ is an unpointed $k$-space, then we write $C_*^{st\A^1}(\sX) $ for $\widetilde{C}_*^{st\A^1}(\sX_+)$. The category $D_{\A^1}(k)$ has an unit object, denoted by $\bold{1}_k$, which is the complex $\widetilde{C}^{st\A^1}_*(\bold{S}^0)$. Define $\bold{1}_k[n] = \bold{1}_k \otimes \widetilde{C}_*^{st\A^1}(S^n_s) $ and $\widetilde{C}_*^{st\A^1}(\sX)[n] = \widetilde{C}_*^{st\A^1}(\sX) \otimes \bold{1}_k[n]$ for a $k$-space $(\sX,x) \in Spc_+/k$. 
\begin{defn}\label{defA1sthomology}
Let $\sX \in Spc/k$ be a $k$-space. The $i$-th $\P^1$-stable $\A^1$-homology sheaf $\bold{H}^{st\A^1}_i(\sX)$ is the Nisnevich sheaf associated to the presheaf 
$$U \mapsto \Hom_{D_{\A^1}(k)}(C_*^{st\A^1}(U)[i],C_*^{st\A^1}(\sX)). $$
\end{defn}     
Just like in case of stable $\A^1$-homotopy categories, one has the following result 
\begin{prop}\label{propstabilize}\cite[Prop. 2.1.29]{AH11}
Let $U \in Sm/k$ and $(\sX,x) \in Spc_+/k$. One has a canonical isomorphism 
\begin{multline}\label{eqstabilize} 
colim_n \Hom_{D^{eff}_{\A^1}(k)}(C_*^{\A^1}(U) \otimes \Z_{\A^1}(n) [i], \widetilde{C}_*^{\A^1}(\sX) \otimes \Z_{\A^1}(n)[i]) \stackrel{\cong}{\rightarrow} \\  \Hom_{D_{\A^1}(k)}(C_*^{st\A^1}(U),\widetilde{C}_*^{st\A^1}(\sX)).
\end{multline}
\end{prop}
The Hurewicz formalism induces the following functors, which one still calls Hurewicz functors (or abelianization functors) 
\begin{align*}
\bold{StHo}_{\A^1,S^1}(k) \rightarrow D^{eff}_{\A^1}(k), \\ \bold{StHo}_{\A^1,\P^1}(k) \rightarrow D_{\A^1}(k),
\end{align*}
which give rise to morphisms of sheaves 
\begin{align*}
\pi_i^{st\A^1,S^1}(\Sigma^{\infty}_s(\sX_+)) \rightarrow \bold{H}_i^{\A^1}(\sX), \\ 
\pi_i^{st\A^1,\P^1}(\Sigma^{\infty}_{\P^1}(\sX_+)) \rightarrow \bold{H}_i^{st\A^1}(\sX). 
\end{align*}
\begin{defn}\label{defnstA1coh}
Let $\sX \in Spc/k$ be a $k$-space. The bigraded $\P^1$-stable $\A^1$-cohomology group $H^{p,q}_{st\A^1}(\sX,\Z)$ is defined as 
$$H^{p,q}_{st\A^1}(\sX,\Z) = \Hom_{D_{\A^1}(k)}(C_*^{st\A^1}(\sX),\Z_{\A^1}(q)[p]).$$ 
\end{defn}
The advantage of $\P^1$-stable $\A^1$-derived category $D_{\A^1}(k)$ is that one has duality formalism. In the context of stable $\A^1$-homotopy theory, it was done in \cite[App. A]{Hu05}. By using the canonical map  
$$\bold{1}_k \rightarrow \Sigma^{\infty}_{\P^1}Th(-T_X),$$ 
one can show
\begin{prop}\label{propdual}\cite[Prop. 3.5.2 and Lem. 3.5.3]{AH11}
Let $X \in SmProj/k$, then $C_*^{st\A^1}(X)$ is a strong dualizable object in $D_{\A^1}(k)$ and its dual is $C_*^{st\A^1}(X)^{\vee} = \widetilde{C}_*^{st\A^1}(Th(-T_X))$. Consequently, one has a canonical isomorphism 
\begin{equation}\label{eqdual}
\Hom_{D_{\A^1}(k)}(\bold{1}_k,C_*^{st\A^1}(X)) \stackrel{\cong}{\rightarrow} \Hom_{D_{\A^1}(k)}(C_*^{st\A^1}(X)^{\vee},\bold{1}_k). 
\end{equation}
\end{prop}
We end up this section by a definition 
\begin{defn}\label{defgmA1}
Let $k$ be a field. One defines the geometric stable $\A^1$-derived category $D_{\A^1,gm}(k)$ over $k$ as the thick subcategory of $D_{\A^1}(k)$ generated by $C_*^{st\A^1}(X)$ for $X \in SmProj/k$.
\end{defn}
\section{Proof of Morel's conjecture}
\subsection{Homotopy $t$-structure and motivic cellular spectra}
We recall in this section the notion homotopy $t$-structure in terms of generators (cf. \cite{Ay08}) and prove then the conjecture of F. Morel \ref{conjst}. Let $S$ be a Noetherian scheme of finite Krull dimension, which we refer as base scheme. Let us abbreviate the $\P^1$-stable motivic homotopy category over $S$ simply by $SH(S)$ and identify its $\Q$-localization with $SH(S)_{\Q} \simeq D_{\A^1}(S)_{\Q}$ (\cite{CD10}). The subcategory $SH(S)_{\geq n}$ is generated under homotopy colimits and extensions by 
$$\{S^{p,q} \wedge \Sigma^{\infty}_{\P^1}(X_+) | X \in Sm/S, p-q \geq n \},$$
where $S^{p,q} = S^{p-q}_s \wedge S_t^{q}$ denotes the motivic spheres. We set 
$$SH(S)_{\leq n} = \{ E \in SH(S)| [F,E] = 0, \forall F \in SH(S)_{\geq n+1} \}$$ 
The bigraded motivic homotopy sheaves are defined as 
$$\underline{\pi}_{p,q}^{st\A^1}(E) = a_{Nis}(U \mapsto [S^{p,q}\wedge \Sigma^{\infty}_{\P^1}(U_+),E].$$
The following description of the homotopy $t$-structure is a consequence of F. Morel's stable $\A^1$-connectivity result (see for instance \cite{Mor04a}): 
\begin{thm}[F. Morel]
Let $S = \Spec k$, where $k$ is a field. 
\begin{enumerate}
\item The triple $(SH(k),SH(k)_{\geq 0},SH(k)_{\leq 0})$ is a $t$-structure on $SH(k)$.
\item The heart of the homotopy $t$-structure $\pi_*^{\A^1}(k) = SH(k)_{\geq 0} \cap SH(k)_{\leq 0}$ is identified with the category of homotopy modules. 
\item The homotopy $t$-structure is non-degenerated in the sense that for any $U \in Sm/k$ and any $E \in SH(k)$, one has the morphism 
$$[\Sigma^{\infty}_{\P^1}(U_+),E_{\geq n}] \rightarrow [\Sigma^{\infty}_{\P^1}(U_+),E]$$ 
is an isomorphism for $n \leq 0$ and the morphism 
$$[\Sigma^{\infty}_{\P^1}(U_+),E] \rightarrow [\Sigma^{\infty}_{\P^1}(U_+),E_{\leq n}]$$ 
is an isomorphism for $n > \dim (U)$.
\end{enumerate}
\end{thm}
From this one can deduce easily that $E \in SH(k)_{\geq n}$ iff $\pi_{p,q}^{st\A^1}(E) = 0$ for $p-q < n$ and $E \in SH(k)_{\leq n}$ iff $\pi_{p,q}^{st\A^1}(E) = 0$ for $p-q > n$. Now given a motivic spectrum $E \in SH(S)$, where $S$ is a base scheme, we may consider the Postnikov's co-tower 
$$\xymatrix{\cdots \ar[r] &  E_{\geq q+1} \ar[r] & E_{\geq q} \ar[r] \ar[d] & E_{\geq q-1} \ar[r] \ar[d] & \cdots \\ & & \Sigma_s^q H\pi_{q}^{st\A^1}(E)_* & \Sigma_s^{q-1}H\pi_{q-1}^{st\A^1}(E)_* }$$
By taking homotopy we obtain a long exact sequence 
$$\cdots \rightarrow \underline{\pi}^{st\A^1}_{p,n}(E_{\geq q+1}) \rightarrow \underline{\pi}^{st\A^1}_{p,n}(E_{\geq q}) \rightarrow \underline{\pi}^{st\A^1}_{p,n}(S^q_s \wedge H \pi^{st\A^1}_q(E)_*) \rightarrow \underline{\pi}^{st\A^1}_{p-1,n}(E_{\geq q+1}) \rightarrow \cdots $$
We set $D^1_{p,q,n} = \underline{\pi}_{p,n}^{st\A^1}(E_{\geq q})$ and $E^1_{p,q,n} = \underline{\pi}_{p,n}^{st\A^1}(S^q_s \wedge H \pi^{st\A^1}_q(E)_*)$. This gives rise to an exact couple 
$$\xymatrix{D^1_* \ar[rr]^{(0,-1,0)} && D^1_* \ar[ld]^{(0,0,0)} \\ & E^1_* \ar[lu]^{(-1,1,0)} }$$
Thus we obtain a spectral sequence for a smooth scheme $X/S$ 
\begin{equation}\label{eqssPostnikov}
E^1_{p,q,n} = \underline{\pi}^{st\A^1}_{p,n}(S^q_s \wedge H\pi_{q}^{st\A^1}(E)_*)(X) \Rightarrow \underline{\pi}^{st\A^1}_{p,n}(E)(X),  
\end{equation}
where the differential $d^1_{p,q,n}: E^1_{p,q,n} \rightarrow E^1_{p-1,q+1,n}$ of the spectral sequence is induced by the composition 
$$S^q_s \wedge H\pi^{st\A^1}_q(E)_* \rightarrow S^1_s \wedge E_{\geq q+1} \rightarrow S^1_s \wedge (S^{q+1}_s \wedge H\pi_{q+1}^{st\A^1}(E)_*).$$
The $r$-th differential of the spectral sequence \ref{eqssPostnikov} is given by $d^r_{p,q,n}: E^r_{p,q,n} \rightarrow E^r_{p-1,q+r,n}$.
\begin{prop}\label{propssPostnikov}
Let $S \rightarrow \Spec k$ be an essentially smooth $k$-scheme and $E \in SH(S)$ be a motivic spectrum. Then the spectral sequence \ref{eqssPostnikov} converges strongly.
\end{prop}
\begin{proof}
This follows from the non-degeneration of the homotopy $t$-structure (cf. \cite[Rem. 2.6]{Hoy13} and \cite[Lem. 6.3.1 and 6.3.2]{Mor05}). 
\end{proof}

\begin{defn}\label{defnunstcell} \cite{DI05}(unstably cellular){ \rm
Let $\sA = \{S^{p,q}| p \geq q \geq 0\}$ be the set of motivic spheres in $\mathbf{Ho}_{\A^1,+}(k)$. An object $X \in \mathbf{Ho}_{\A^1,+}(k)$ is called unstably, if it is $\sA$-cellular,
 which means: it belongs to the smallest class of objects of $\mathbf{Ho}_{\A^1,+}(k)$, which satisfies 
\begin{enumerate}
\item every object of $\sA$ is $\sA$-cellular;
\item if $X$ is weak equivalent to an $\sA$-cellular object, then $X$ is $\sA$-cellular;
\item if $D: I \rightarrow \mathbf{Ho}_{\A^1,+}(k)$ is a diagram such that each $D_i$ is $\sA$-cellular, then so is $hocolim D$.
\end{enumerate}  
} 
\end{defn}

\begin{lem}\label{lemunstcell}\cite[Lem. 2.2]{DI05}
Let $X \rightarrow Y \rightarrow Z$ be a homotopy cofiber sequence in $\mathbf{Ho}_{\A^1,+}(k)$. If $X$ and $Y$ are unstably cellular, then $Z$ is also unstably cellular. 
\end{lem}

\begin{defn}\label{defncell} \cite{DI05}(stably cellular) {\rm
Let $\sS = \{ S^{p,q}| p,q \in \Z \}$ the class of all motivic spheres in $SH(k)$. A spectrum $E \in SH(k)$ is called cellular, if $E$ is $\sS$-cellular, i.e. it belongs to the smallest class of objects of $SH(k)$, which satisfies
\begin{enumerate}
\item every object of $\sS$ is $\sS$-cellular;
\item if $X$ is weak equivalent to an $\sS$-cellular object, then $X$ is $\sS$-cellular;
\item if $D: I \rightarrow SH(k)$ is a diagram such that each $D_i$ is $\sS$-cellular, then so is $hocolim D$.
\end{enumerate}
}
\end{defn}
We need one more definition: 
\begin{defn}\label{defncell1}\cite{DI05}{ \rm
An object $X \in \mathbf{Ho}_{\A^1,+}(k)$ is called stably cellular, if $\Sigma^{\infty}_{\P^1}(X) \in SH(k)$ is cellular.

}
\end{defn}
The connection between these two notions of cellularity is the following:
\begin{lem}\label{lemconncell}\cite[Lem. 3.1]{DI05}
Let $X$ be an unstably cellular object in $\mathbf{Ho}_{\A^1,+}(k)$, then $\Sigma^{\infty}_{\P^1}(X) \in SH(k)$ is cellular. 
\end{lem}
We recollect some facts in \cite{DI05} about the cellularity of spectra 
\begin{prop}\label{propcell}\cite[Lem. 2.5, Lem. 3.4]{DI05}
The followings are true 
\begin{enumerate}
\item If $E' \rightarrow E \rightarrow E'' \stackrel{+1}{\rightarrow}$ is a homotopy cofiber sequence in $SH(k)$, such that any two of $E', E$ and $E''$ are cellular, then so is the third.  
\item If $E$ and $F$ are cellular objects in $SH(k)$, then so are $E \wedge F$ and $E \times F$.
\end{enumerate}
\end{prop}
In the following we denote by $\mathbf{KO}$ the $\P^1$-spectrum representing the Hermitian $K$-theory constructed in \cite[\S 5]{Horn05}. It is defined by: for all $k \in \mathbb{N}$, $\mathbf{KO}_{4k} = a_{Nis}KO^{fib}$, $\mathbf{KO}_{4k+1} = a_{Nis-}U^{fib}$, $\mathbf{KO}_{4k+2} = a_{Nis-}KO^{fib}$ and $\mathbf{KO}_{4k+3} = a_{Nis}U^{fib}$, where $a_{Nis}$ denotes the Nisnevich sheafification functor, $(-)^{fib}$ a fibrant replacement wrt. the simplicial model structure defined by Jardine \cite{Jar87}. We need the following:
\begin{lem}\label{lemcell}\cite[Lem. 6.1]{DI05}
Let $E = \{E_n, \P^1 \wedge E_n \rightarrow E_{n+1}\}$ be a motivic spectrum. Then $E$ is weakly equivalent to 
$$hocolim \{\Sigma^{\infty}_{\P^1}E_0 \rightarrow \Sigma^{-2,-1}_{\P^1} \Sigma^{\infty}_{\P^1} E_1 \rightarrow \Sigma^{-4,-2}_{\P^1}\Sigma^{\infty}_{\P^1}E_2 \rightarrow \cdots \} $$ 
\end{lem}
We now show:
\begin{thm}\label{thmKO}
Let $k$ be a field of $char(k) \neq 2$. The $\P^1$-spectrum $\mathbf{KO}$ is stable cellular.
\end{thm}
\begin{proof} 
By lemma \ref{lemcell} above, we need to show that $\mathbf{KO}_i$ is cellular for every $i \in \N$. From the work of M. Schlichting and G. S. Tripathi \cite[Thm. 1.3]{ST13} we know that in the unstable $\A^1$-homotopy category $\mathbf{Ho}_{\A^1,+}(k)$ one has
$$\mathbf{KO}_i  =  \begin{cases} \Z \times GrO_{\bullet}, i = 0 \\ Sp/GL, i= 1 \\ \Z \times BSp, i = 2 \\ O/GL, i = 3  \end{cases}$$ 
By definition, the infinite orthogonal Grassmannian is the presheaf 
$$GrO_{\bullet} = colim_{V \subset H^{\infty}}GrO_{rk(V)}(V \perp H^{\infty}),$$
where $H^{\infty} = colim_n H^n$ denotes the infinite hyperbolic space and the colimit for the orthogonal Grassmannians is taken over the non-degenerated subbundles $V \subset H^{\infty}$. It is difficult to prove directly the cellularity of the orthogonal Grassmannians. However, we have a homotopy fiber sequence 
$$U \rightarrow KGL \stackrel{H}{\rightarrow} KO.$$
From the work of \cite[Prop. 4.4, Thm. 6.2]{DI05} we know $KGL$ is cellular. So with the homotopy fiber sequence above, we may reduce the problem to prove the cellularity of $U$-theory, which is $O/GL$. Again, by \cite[Prop. 4.1]{DI05}, the linear algebraic group $GL_n$ is cellular, so $GL = colim_n GL_n$ is also cellular, since this colimit is filtered colimit. So it is enough to see that $O$ is cellular. We have the following cofiber sequence 
$$O_{2n-2} \inj O_{2n} \surj Q_{2n},$$
where $Q_{2n}$ denotes the affine quadrics $\{x_1 \cdot y_1 + \cdots + x_n \cdot y_n  = 1\}$. Here we identify $Q_{2n}$ as homotopy cofiber of the canonical inclusion $O_{2n-2} \inj O_{2n}, \quad A \mapsto A \oplus 1$ in the homotopy pushout diagram 
$$\xymatrix{ O_{2n-2} \ar[d] \ar@{^{(}->}[r] & O_{2n} \ar[d] \\ \{pt\} \ar[r] & Q_{2n} }$$
But there is a canonical bundle structure 
$$Q_{2n} \rightarrow \A^n - \{ 0 \}, \quad (x_1,\cdots,x_n,y_1,\cdots ,y_n) \mapsto (y_1,\cdots, y_n) $$
with fiber $\A^{n-1}$. This proves the cellularity of $O$ in $\mathbf{Ho}_{\A^1,+}(k)$ by means of induction (use Lemma \ref{lemunstcell}). With the same argument, it remains to show the cellularity of $Sp = colim_n Sp_{2n}$ in $\mathbf{Ho}_{\A^1,+}(k)$, but this also follows from induction and the cofiber sequence 
$$Sp_{2n-2} \inj Sp_{2n} \surj \A^{2n} - \{ 0 \},$$
where we again identify $\A^{2n}-\{ 0\}$ with the homotopy cofiber of the inclusion $Sp_{2n-2} \inj Sp_{2n}, \quad M \mapsto M \oplus 1$ in the homotopy pushout diagram 
 $$\xymatrix{ Sp_{2n-2} \ar[d] \ar@{^{(}->}[r] & Sp_{2n} \ar[d] \\ \{pt\} \ar[r] & \A^{2n} -\{0 \} }$$
Apply now the Lemma \ref{lemconncell} we are done. \end{proof}
Let $S$ be any scheme. We recall now the rationally splitting of F. Morel in $SH(S)$ (see \cite[\S 16.2]{CD10}). The permutation isomorphism 
$$\tau: \Sigma^{\infty}_{\P^1,+} \G_{m,\Q} \wedge \Sigma^{\infty}_{\P^1,+}\G_{m,\Q} \rightarrow  \Sigma^{\infty}_{\P^1,+} \G_{m,\Q} \wedge \Sigma^{\infty}_{\P^1,+}\G_{m,\Q}$$
satisfies $\tau^2 = 1$. This defines an element $e \in End_{SH(S)_{\Q}}(\mathbf{1}_{\Q})$, such that $e^2 = 1$. So we may define 
$$e_+ = \frac{e-1}{2}, \quad e_- = \frac{e+1}{2}.$$
Then we define $\mathbf{1}_{\Q+} = \im(e_+)$ and $\mathbf{1}_{\Q-} = \im(e_-)$. For any spectrum $\sE \in SH(S)_{\Q}$, one defines $\sE_+ = \mathbf{1}_{\Q+} \wedge \sE$ and $\sE_- = \mathbf{1}_{\Q-} \wedge \sE$. This leads to a splitting of stable homotopy category 
$$SH(S)_{\Q+} \times SH(S)_{\Q-} \stackrel{\simeq}{\rightarrow} SH(S)_{\Q}, \quad (\sE_+,\sE_-) \mapsto \sE_+ \wedge \sE_-$$
\begin{lem}\label{lemcellsummand}
Let $E \in SH(S)$ be a cellular motivic spectrum. If $F$ is a direct summand of $E$, then $F$ is a also cellular.
\end{lem}
\begin{proof}
Suppose $E = F \vee E'$. Let $p: E \rightarrow E$ be the projector onto $F$. The we have 
$$hocolim (E \stackrel{p}{\rightarrow} E \stackrel{p}{\rightarrow} E \stackrel{p}{\rightarrow} \cdots) = F.$$
So if $E$ is cellular, then $F$ is also cellular.
\end{proof}
\begin{lem}\label{lemcellHW}
Let $k$ be an abitrary field of $char(k) \neq 2$. Then the $\A^1$-Eilenberg-Maclane spectrum $H\mathbf{W}_{\Q}$ associated to the unramified Witt sheaf $\mathbf{W}_{\Q}$ is stable cellular.
\end{lem} 
\begin{proof}
From Morel's computation (cf. \cite[Cor. 25, Rem. 26]{Mor12}) we have a distinguished triangle 
$$(\mathbf{1}_{\Q-})_{> 0} \rightarrow \mathbf{1}_{\Q-} \rightarrow H\mathbf{W}_{\Q} \stackrel{+1}{\rightarrow}.$$
By construction we have $\mathbf{1}_{\Q-} = (\mathbf{1}_{\Q-})_{>0} \vee H\mathbf{W}_{\Q}$. So the lemma \ref{lemcellsummand} above shows that $H \mathbf{W}_{\Q}$ is also stably cellular.
\end{proof}
As a consequence of the theorem \ref{thmKO}, we obtain: 
\begin{cor}\label{corHW}
Let $k$ be a perfect field of $char(k) \neq 2$. The spectrum $\mathbf{KT}_{\Q}$, which represents the triangular Balmer's Witt group, is stably cellular. Moreover, there is a splitting 
$$\mathbf{KT}_\Q = \bigvee_{n \in \Z} S^{4n}_s \wedge H\mathbf{W}_\Q.$$ 
\end{cor}
\begin{proof}
Let $\mathbf{KT}$ denote the $\P^1$-spectrum constructed in \cite{Horn05}. One has 
$$\mathbf{KT} = \mathbf{KO}[\eta^{-1}] = hocolim \{\mathbf{KO} \stackrel{\eta}{\rightarrow} \mathbf{KO}\wedge \G_m^{-1} \stackrel{\eta}{\rightarrow} \mathbf{KO} \wedge \G_m^{\wedge -2} \rightarrow \cdots \}.$$
So theorem \ref{thmKO} implies that the spectrum $\mathbf{KT}$ is also cellular. Let $X \in Sm/k$ be a smooth $k$-scheme and consider the spectral sequence \ref{eqssPostnikov} for the presheaf $\pi_{*,*}(\mathbf{KT}_\Q)(X) \stackrel{defn}{=} [S^{*,*}\wedge \Sigma^{\infty}_{\P^1_\Q,+}(X), \mathbf{KT}_\Q]$, we have a strongly convergent spectral sequence
\begin{equation}\label{eqSSKTQ}
E^1_{p,q,n} = \pi_{p,n}(H\pi_q(\mathbf{KT}_\Q)_*)(X) \Rightarrow \pi_{p,n}(\mathbf{KT}_\Q)(X).   
\end{equation}
We know that the stable Hopf map $\eta: \G_{m,\frac{1}{2}-} \rightarrow \mathbf{1}_{\frac{1}{2}-}$ is an isomorphism so we have $\mathbf{KT}_\Q \cong \mathbf{KT}_{\Q-} \cong \mathbf{KO}_{\Q-}$ and $\pi_{p,n}(\mathbf{KT}_\Q)(X) = \pi_{p-n}(\mathbf{KT}_\Q)(X)$. By the work of \cite{Horn05} we know $\pi_{p-n}(\mathbf{KT}_\Q)(X) = W^{n-p}_B(X)_\Q$. So we have $H\pi_q(\mathbf{KT}_\Q) = H\mathbf{W}_\Q$ for $q \equiv 0 \mod 4$ and $0$ else. In this way, we obtain 
$$E^1_{p,q,n} = \begin{cases} H^{n-p}_{Nis}(X,\underline{\mathbf{W}}_\Q), \quad q \equiv 0 \mod 4, \\ 0, \quad \text{else}.  \end{cases}$$  
Clearly, for any $X \in Sm/k$ we have $d^r_{p,q,n}(X) = 0$ for $r \not \equiv 0 \mod 4$. The differentials $d^r_{p,q,n}(X)$ is induced by 
$$ S^q_s \wedge H\mathbf{W}_\Q \rightarrow S^1_s \wedge (S^{q+r}_s \wedge H\mathbf{W}_\Q),$$
which lives in $H\mathbf{W}_\Q^{r+1}(H \mathbf{W}_\Q)$. Since $H\mathbf{W}_\Q$ is cellular from the lemma \ref{lemcellHW}, we may apply \cite[Prop. 7.7]{DI05} to have a conditionally cohomological tri-graded spectral sequence 
$$E^{a,(b,c)}_2 = Ext_{\underline{\pi}_{*,*}(\mathbf{1}_{\Q-})(k)}^{a,(b,c)}(\underline{\pi}_{*,*}(H\mathbf{W}_\Q)(k),\underline{\pi}_{*,*}(H\mathbf{W}_\Q)(k)) \Rightarrow H\mathbf{W}^{a+b,c}(H\mathbf{W}_\Q).$$
As in the proof of the lemma \ref{lemcellHW} we have $H\mathbf{W}_{\Q}$ is a direct summand of $\mathbf{1}_{\Q-}$, so we have $E^{a,(b,c)}_2 = 0$ for $a > 0$. So all the terms $\lim^1 E^{***}_r$ vanish, so the spectral sequence is strongly convergent and degenerates, so we have $H\mathbf{W}_\Q^*(H\mathbf{W}_\Q) = 0$ for $* >0$. This implies that the spectral sequence \ref{eqSSKTQ} degenerates for any smooth $k$-scheme $X$ and hence we obtain the splitting of $\mathbf{KT}_{\Q}$.
\end{proof}
\begin{cor}\label{corssgwdeg}
Let $X \in Sm/k$ be a smooth $k$-scheme, where $k$ is a perfect field of $char(k) \neq 2$. The Gersten-Witt spectral sequence \cite[Thm. 7.2]{BW02} 
$$W_B^{p+q}(X) \Leftarrow E^{p,q}_1 = \begin{cases} \bigoplus_{x \in X^{(p)}}W(\kappa(x)), \quad 0 \leq p \leq \dim(X), \, q \equiv 0 \mod 4, \\ 0, \quad \text{otherwise},   \end{cases} $$ 
degenerates rationally at $E_2$-terms. 
\end{cor} 
\begin{proof}
This follows immediately from the splitting $\mathbf{KT}_{\Q} \cong \bigvee_{n \in \Z}S^{4n}_s \wedge H\mathbf{W}_{\Q}$ in $SH(k)_{\Q}$.
\end{proof}
\subsection{Functoriality in motivic stable homotopy}
Following \cite{Ay08}, we recall that the stable homotopy category of schemes defines a $2$-functor from category of schemes $Sch$ to the category of symmetric monoidal closed triangulated categories. So we have a bunch of axioms: for any morphism of schemes $f: T \rightarrow S$, there is a pullback functor 
$$f^*: SH(S) \rightarrow SH(T),$$
such that $(f \circ g)^* = g^* \circ f^*$. Moreover, 
\begin{enumerate}
\item One has an adjunction for any morphism of schemes $f: T \rightarrow S$ 
$$f^*: SH(S) \leftrightarrows SH(T) : f_*.$$
If $f$ is smooth, then one has an adjunction 
$$f_{\#}: SH(T) \leftrightarrows SH(S) : f^* $$
\item Given a cartesian square 
\begin{equation*} \xymatrix{ Y \ar[r]^q \ar[d]_g & X \ar[d]^f \\ T \ar[r]_p & S}
\end{equation*}
and assume $f$ is smooth, then 
$$f_{\#}p^* \stackrel{\cong}{\rightarrow} g_{\#}q^*$$
\item Let $f: Y \rightarrow X$ be a smooth morphism, $\sE \in SH(Y)$ and $\sF \in SH(X)$, the natural transformation 
$$f_{\#}(\sE \wedge f^* \sF) \stackrel{\cong}{\rightarrow} f_{\#}\sE \wedge \sF$$ 
is an isomorphism.
\item Let $i: Z \inj X$ be a closed immersion with complement $j: U \inj X$, then there is a distinguished triangle 
$$j_{\#}j^* \rightarrow Id \rightarrow i_*i^* \stackrel{+1}{\rightarrow}$$
\item For any closed immersion $i: Z \inj X$, one has an adjunction 
$$i_*: SH(Z) \leftrightarrows SH(X) : i^! $$
\item Given a cartesian square 
\begin{equation*} \xymatrix{ T \ar[r]^k \ar[d]_g & Y \ar[d]^f \\ Z \ar[r]_i & X}
\end{equation*}
where $i: Z \inj X$ is a closed immersion, then one has an isomorphism 
$$f^*i_* \stackrel{\cong}{\rightarrow} k_*g^*$$ 
\item Let $i: Z \inj X$ be a closed immersion, $\sE \in SH(Z)$ and $\sF \in SH(X)$, the natural transformation 
$$i_*(\sE \wedge i^*\sF) \stackrel{\cong}{\rightarrow} i_* \sE \wedge \sF$$
is an isomorphism.
\end{enumerate}
\begin{lem}\label{lemstofHW}
Let $S$ be a Noetherian scheme of finite Krull dimension. Then $\mathbf{1}_{\Q-,S}$ is stable under pullback functor, i.e. for any morphism of schemes $f: T \rightarrow S$ one has $f^*(\mathbf{1}_{\Q-,S}) = \mathbf{1}_{\Q-,T}$. Assume moreover $f: S \rightarrow \Spec k$ is essentially smooth over a field $k$ of $char(k) \neq 2$, then $H\mathbf{W}_{\Q,S}$ is also stable under pullback functor $f^*$, where $H\mathbf{W}_{\Q,S}$ denotes the $\A^1$-Eilenberg-Maclane spectrum associated to the unramified Witt sheaf $\underline{\bold{W}}_{\Q}$ over $S$.  
\end{lem}
\begin{proof}
For any Noetherian scheme of finite Krull dimension $S$ one has $\mathbf{1}_{\Q+,S} \cong H_{B,S}$, which is stable under pullback (\cite[Thm. 16.2.13]{CD10}) and pullback functor commutes with direct sum, this implies $\mathbf{1}_{\Q-,S}$ is stable under pullback functor. Assume now $f: S \rightarrow \Spec k$ is essentially smooth over a field $k$ of $char(k) \neq 2$. Since $f: S \rightarrow \Spec k$ is essentially smooth, so by \cite[Lem. 6.3.2, Lem. 6.3.1]{Mor05} we may conclude $\mathbf{1}_S$ is $-1$-connected. By definition, we have a distinguished triangle 
$$(\mathbf{1}_{\Q-,S})_{\geq 1} \rightarrow \mathbf{1}_{\Q-,S} \rightarrow H\pi_{0}(\mathbf{1}_{\Q-,S})_* \stackrel{+1}{\rightarrow}.$$
Now we have (see for instance \cite[Appendix]{Hoy13})
\begin{multline*}
[S^{p,q} \wedge \Sigma^{\infty}_{\P^1_S,+}(U),\mathbf{1}_S]_{SH(S)} = colim_{\alpha}[S^{p,q} \wedge \Sigma^{\infty}_{\P^1_{S_\alpha},+}(U_{\alpha}),f^*_{\alpha}\mathbf{1}_k] = \\ colim_{\alpha} [S^{p,q} \wedge (f_{\alpha})_{\#}\Sigma^{\infty}_{\P^1_{S_\alpha},+}(U_{\alpha}),\mathbf{1}_k]  = colim_{\alpha}[S^{p,q} \wedge \Sigma^{\infty}_{\P^1,+}(U_\alpha),\mathbf{1}_k]_{SH(k).}
\end{multline*} 
Now the result follows from $f^*H\pi_0(\mathbf{1}_{\Q-})_* = Hf^*\pi_0(\mathbf{1}_{\Q-})_* = H \pi_0(f^*\mathbf{1}_{\Q-})_* = H\pi_0(\mathbf{1}_{\Q-,S})_*$
\end{proof}
\begin{rem}{\rm
The proof of the lemma \ref{lemstofHW} doesn't work over a general base scheme, since it relies on stable $\A^1$-connectivity \cite{Mor05}, which is wrong over a general base scheme (see \cite{Ay06}). One may try to generalize the above lemma to a Noetherian regular base scheme $S$ of finite Krull dimension over $\Spec \Z[\frac{1}{2}]$ in a different way as follow: One has the geometric $\A^1$-representability of $\mathbf{KO}$ from the work of \cite{ST13} over $S$. The notion of cellularity in \cite{DI05} can be carried over a base scheme $S$. In fact, as in case of the algebraic $K$-theory spectrum $\mathbf{KGL}$ (see e.g. \cite[Prop. 3.8]{C13}), we may show $\mathbf{KO}$ is stable under pullback functor induced by any morphism of schemes. Of course, the Gersten-Witt spectral sequence in \cite{BW02} is available for any regular scheme $X$ with $\frac{1}{2} \in \Gamma(X,\sO_X)$. Hence, if we know that this spectral sequence degenerates rationally at $E_2$-pages, we may conclude $H\mathbf{W}_{\Q,S}$ is a direct summand of $\mathbf{KT}_{\Q,S} = \mathbf{KO}_{\Q,S}[\eta^{-1}]$, which must be stable under any fullback, without having to know $[\mathbf{1}_S \wedge \G_m^{\wedge *}, \mathbf{1}_S]_{SH(S)}$ explicitly. In case of $\mathbf{KGL}_\Q$, by using slice spectral sequence, one can prove the degeneration of the motivic Atiyah-Hirzebruch spectral sequence rationally at $E_2$-terms, at least over a perfect field $k$. This doesn't work for the Gersten-Witt spectral sequence, since all the slices $s_q(\mathbf{KT}_\Q)$ are trivial.         
}
\end{rem}
We will need the following result: 
\begin{prop}\label{propCDres}\cite[Prop. 7.2.7]{CD10}
Let $\sV$ be a stable perfect symmetric monoidal model category. Assume that $\mathbf{Ho}(\sV)$ admits a small family $\sG$ of compact generator as a triangulated category. (see \cite[Def. 7.2.3]{CD10} for the notion of perfectness). For any monoid $R$ in $\sV$, the triangulated category $\mathbf{Ho}(R-mod(\sV))$ admits the set $\{ R \otimes^{\mathbb{L}} E | E \in \sG \}$ as a family of compact generators.   
\end{prop}
Before continuing, let's say a few words about the homotopy category of modules over the Eilenberg-Maclane spectrum $H \mathbf{W}$ associated to the unramified Witt sheaf $\underline{\mathbf{W}}$. Following F. Morel \cite{Mor04a}, we consider $\underline{\mathbf{W}}_*$ as a homotopy module, where $\underline{\mathbf{W}}_n = \underline{\mathbf{W}}, \forall n \in \Z$. This means $\underline{\Hom}(\G_m^{\wedge n},\underline{\mathbf{W}}) = \underline{\mathbf{W}}, \forall n \in \Z$. So we can consider the $\P^1$-spectrum $H \mathbf{W}$ as a $S_s^1 \wedge \G_m$-spectrum $H \mathbf{W}$, where the $i$-th term is the simplicial sheaf $K(\underline{\mathbf{W}},i) \in \Delta^{op}Sh_{\bullet}((Sm/S)_{Nis})$ and the bonding maps are given by the composition
\begin{multline*}
K(\underline{\mathbf{W}},i) \wedge S_s^1 \wedge \G_m = K(\underline{\mathbf{W}}_i,i) \wedge S_s^1 \wedge \G_m \rightarrow \\ \rightarrow K(\underline{\mathbf{W}}_i, i+1) \wedge \G_m \rightarrow K(\underline{\mathbf{W}}_{i+1},i+1) = K(\underline{\mathbf{W}},i+1).
\end{multline*}
The symmetric group $\Sigma_n$ acts on $(H\mathbf{W})_n$ by permuting the smash factors of $S^n_s = S^1_s \wedge \cdots \wedge S^1_s$. In this way, we obtain a symmetric ring spectrum $H\mathbf{W}$ in $\mathbf{Spect}^{\Sigma}_{\P^1}(S)$, so we can define $\mathbf{Ho}(H\mathbf{W}-mod)$ the homotopy category of modules over $H\mathbf{W}$. We are going to prove now 
\begin{thm}\label{proofofconjst} 
Let $k$ be a field of characteristic $char(k) \neq 2$ and $S \rightarrow \Spec k$ is an essentially smooth $k$-scheme. There is a splitting $\mathbf{1}_{\Q,S} = H_B \vee H\mathbf{W}_\Q$ in $SH(S)_\Q$. In particular, the rational stable $\A^1$-cohomology splits as  
$$H^{p}_{st\A^1}(S,\Q(q)) \simeq H^p_{\sM}(S,\Q(q)) \oplus H^{p-q}_{Nis}(S,\underline{\bold{W}}_{\Q}) $$
\end{thm}
\begin{proof} 
Clearly, from the lemma \ref{lemstofHW} it is enough to prove the theorem in case $S = \Spec k$. Recall the Morel's splitting \cite{CD10} 
$$SH(k)[\frac{1}{2}] = SH(k)_{\frac{1}{2}+} \times SH(k)_{\frac{1}{2}-}. $$
From the work of Cisinski and Deglise \cite{CD10} one knows $SH(k)_{\Q+} \simeq DM_B(k)$, where the last category denotes the category of Beilinson's motives. This category $DM_B(k)$ is known to be equivalent to the category of Voevodsky's big motives $DM(k)_{\Q}$ (\cite{CD10}). Remark that $\G_{m,\frac{1}{2}-} = \bold{1}_{\frac{1}{2}-}$ in $SH(k)_{\frac{1}{2}-}$. Moreover, we even may assume our field $k$ is just the ground field $\Q$, since in case positive characteristic $\mathbf{1}_{\Q-} = 0$ (cf. \cite[Cor. 16.2.14]{CD10}) and in case $char(k)=0$, which means our field is perfect and we know that for any field extension $L/k$ over a perfect field the morphism $\Spec L \rightarrow \Spec k$ is essentially smooth (see e.g. \cite[Appendix]{Hoy13}). It remains to show 
$$SH(k)_{\Q-} = \bold{Ho}(H\bold{W}_{\Q}-mod),$$
where $\mathbf{Ho}(H\mathbf{W}_{\Q}-mod)$ denote the homotopy category of modules over the rational $\A^1$-Eilenberg-Maclane spectrum $H\mathbf{W}_{\Q}$ associated to the unramified Witt-sheaf $\mathbf{W}_{\Q}$. We will follow the same strategy of the proof for $\Q_+$-part in \cite[Thm. 16.2.13]{CD10}. We prove directly that for any $E \in SH(k)_{\Q}$, then $E_- \stackrel{defn}{=} E \wedge \bold{1}_- $ has an $H\bold{W}_{\Q}$-module structure, where $H\bold{W}_\Q$ denotes the Eilenberg-Maclane spectrum associated to the unramified Witt sheaf $\bold{W}_{\Q}$. As $SH(k)_{\Q-}$ is closed under extension, it is enough to show the assertion for homotopy modules $E \in \pi_*^{\A^1}(k)$, by means of the $\A^1$-Postnikov co-tower 
\[ \xymatrix{\cdots \ar[r] &  E_{\geq n+1} \ar[r] & E_{\geq n} \ar[r] \ar[d] & E_{\geq n-1} \ar[r] \ar[d] & \cdots \\ & & \Sigma_s^n H\pi_{n}^{\A^1}(E)_* & \Sigma_s^{n-1}H\pi_{n-1}^{\A^1}(E)_* }\] 
and the non-degeneration of the homotopy t-structure 
$$hocolim_{n \rightarrow -\infty}(E_{\geq n}) = E$$ and
$$holim_{n \rightarrow + \infty}(E_{\geq n}) = 0.$$
Thus we may assume $E \in \pi^{\A^1}_*(k)$ and the general assertion follows inductively. By Morel's stable $\A^1$-connectivity (\cite{Mor05}) one has $\bold{1}_- \in SH(k)_{\Q, \geq 0}$, so we have a distinguished triangle 
$$(\bold{1}_-)_{\geq 1} \rightarrow \bold{1}_- \rightarrow H\pi_{0}^{\A^1}(\bold{1}_-)_* \stackrel{+1}{\rightarrow}.$$   
By Morel's computation of $\pi_0$ of the motivic sphere spectrum (\cite{Mor04}) one obtains then a canonical map 
$$\bold{1}_- \rightarrow H\bold{W}_{\Q}.$$
Since $E \in \pi_*^{\A^1}(k) = SH(k)_{\geq 0} \cap SH(k)_{\leq 0}$ and $SH(k)_{\geq 0}$ is closed under smash product, so we obtain a canonical map 
$$E_- = E \wedge \bold{1}_- \rightarrow H\pi_{0}^{\A^1}(E \wedge \bold{1}_-)_* \rightarrow H \pi_{0}^{\A^1}(E \wedge H\bold{W}_{\Q})_* = E \otimes^{h} H\bold{W}_{\Q},$$
where $\otimes^h$ denotes the product induced in the category of homotopy modules. By equivalence between $\pi_{*}^{\A^1}(k)$ and the category of homotopy modules, we thus obtain a commutative diagram 
\[\xymatrix {E \wedge \bold{1}_- \ar[d]_{can} \ar[r] \ar[dr] & E \otimes^h \pi_{0}^{\A^1}(\bold{1}_-)_* \ar[d] \\ E \wedge H\bold{W}_{\Q} \ar[r]_{\exists} & E } \]
We now need to show that this map $E \wedge H\mathbf{W}_{\Q} \rightarrow E$ satisfies the axiom of an internal module. As in the proof of \cite[Lem. 16.2.8]{CD10}, we only need to know that $\mathbf{Ho}(H\mathbf{W}_{\Q}-mod)$ forms a thick subcategory of $SH(k)_{\Q}$. To do that, we have to prove that for any $H\mathbf{W}_{\Q}$-module $M$, the map 
$$H\mathbf{W}_{\Q} \wedge M \rightarrow M$$
is an isomorphism in $SH(k)_{\Q}$. This is a natural transformation between exact functors which commutes with small sums and since $SH(k)_{\Q}$ is a compactly generated triangulated category, it is enough to check this for the case $M = H\mathbf{W}_{\Q} \wedge E$ with $E$ a compact object of $SH(k)_{\Q}$ (see proposition \ref{propCDres}). To do that, we need to know that the multiplication map 
$$\mu: H\mathbf{W}_{\Q} \wedge H\mathbf{W}_{\Q} \rightarrow H\mathbf{W}_{\Q} $$
is an isomorphism. From the lemma \ref{lemcellHW}, we know $H\mathbf{W}_{\Q}$ is stable cellular. So we may apply \cite[Prop. 7.7]{DI05} to have a strongly convergent homological tri-graded spectral sequence 
\begin{equation}\label{spectseq1}
E^2_{a,(b,c)} = Tor_{a,(b,c)}^{\pi_{*,*}(\mathbf{1}_{\Q-})}(\pi_{*,*}(H\mathbf{W}_{\Q}),\pi_{*,*}(H\mathbf{W}_{\Q})) \Rightarrow \pi_{a+b,c}(H\mathbf{W}_{\Q} \wedge H\mathbf{W}_{\Q}). 
\end{equation}
As before in the proof of the lemma \ref{lemcellHW} one has a splitting $\mathbf{1}_{\Q-} = (\mathbf{1}_{\Q-})_{\geq 1} \vee H\mathbf{W}_{\Q}$, which means that $H\mathbf{W}_{\Q}$ is a direct summand of $\mathbf{1}_{\Q-}$, so $E^2_{a,(b,c)} = 0$ for $a \neq 0$. This shows that 
$$\pi_{*,*}(H\mathbf{W}_\Q \wedge H\mathbf{W}_\Q) \cong \pi_{*,*}(H\mathbf{W}_\Q) \otimes_{\pi_{*,*}(\mathbf{1}_{\Q-})} \pi_{*,*}(H\mathbf{W}_{\Q})$$ and $(\pi_{*,*}(H\mathbf{W}_\Q),\pi_{*,*}(H\mathbf{W}_\Q \wedge H\mathbf{W}_\Q))$ is a flat Hopf-algebroid, where the term Hopf-algebroid implies that the composition
$$\pi_{*,*}(H\mathbf{W}_\Q) \stackrel{\circ (\id \wedge \eta_{H\mathbf{W}_\Q})}{\longrightarrow} \pi_{*,*}(H\mathbf{W}_\Q \wedge H\mathbf{W}_\Q) \stackrel{\circ \mu_{H\mathbf{W}_\Q}}{\longrightarrow} \pi_{*,*}(H\mathbf{W}_\Q)$$
is the identity. It remains to show now that the projector 
$$\Phi: H\mathbf{W}_\Q \wedge H\mathbf{W}_\Q \stackrel{\mu_{H \mathbf{W}_\Q}}{\longrightarrow}H\mathbf{W}_\Q \stackrel{\id \wedge \eta_{H \mathbf{W}_\Q}}{\longrightarrow} H \mathbf{W}_\Q \wedge H\mathbf{W}_\Q$$
is also just the identity. Since $H\mathbf{W}_\Q$ is stable cellular, the spectrum $H\mathbf{W}_\Q \wedge H\mathbf{W}_\Q$ is also stable cellular. Hence we can apply \cite[Prop. 7.7]{DI05} to have a conditionally convergent cohomological tri-graded spectral sequence 
\begin{multline*}
E_2^{a,(b,c)} = Ext^{a,(b,c)}_{\underline{\pi}_{*,*}(\mathbf{1}_{\Q-})(k)}((H\mathbf{W}_\Q \wedge H\mathbf{W}_\Q)_{*,*},(H\mathbf{W}_\Q \wedge H\mathbf{W}_\Q)_{*,*}) \Rightarrow \\ (H\mathbf{W}_\Q \wedge H\mathbf{W}_\Q)^{a+b,c}(H\mathbf{W}_\Q \wedge H\mathbf{W}_\Q). 
\end{multline*} 
As before $E^{a,(b,c)}_2 = 0$ for $a > 0$, because we calculate above that $\pi_{*,*}(H\mathbf{W}_\Q \wedge H\mathbf{W}_\Q)$ is again a projective module over $\pi_{*,*}(\mathbf{1}_{\Q-})$. So this spectral sequence degenerates at its $E_2$-page, which implies $\lim^1 E^{***}_r = 0$. Hence the spectral sequence converges strongly. So we have now 
$$H\mathbf{W}_\Q \wedge H\mathbf{W}_\Q(H\mathbf{W}_\Q \wedge H\mathbf{W}_\Q) \cong \Hom_{\underline{\mathbf{W}}_\Q(k)}(\underline{\mathbf{W}}_\Q(k),\underline{\mathbf{W}}_\Q(k)) = \underline{\mathbf{W}}_\Q(k).$$
Now since $k$ is just the ground field $\Q$, we have $\mathbf{W}_\Q(k) = \Q$. So we have $\Phi = \lambda \cdot \id$ for $\lambda \in \Q$. The case $\lambda = 0$ can not happen because $\Phi$ is a projector on a non-trivial factor. This proves then our theorem.  
\end{proof}
\begin{cor}\label{corsplitKT}
Let $S \rightarrow \Spec k$ be an essentially smooth scheme over a field $k$ of $char(k) \neq 2$. There is a splitting in $SH(S)_\Q$
$$\mathbf{KT}_{\Q,S} = \mathbf{KT}_{\Q-,S} = \mathbf{KO}_{\Q-,S} = \bigvee_{n \in \Z} S^{4n}_{s,\Q-}.$$ 
\end{cor}
\begin{proof}
This follows immediately from the proof of corollary \ref{corHW} and theorem \ref{proofofconjst}.
\end{proof}
We also obtain the following corollary, which is a motivic analog to topology, saying that all higher stable homotopy groups of the sphere spectrum are torsion.  
\begin{cor}(Motivic Serre's theorem). \label{cormotivicserre}
Let $k$ be a perfect field of $char(k) \neq 2$. For every $n > 0$ the stable homotopy groups of the rationally motivic sphere spectrum $\pi_{n,0}(\mathbf{1}_\Q) = 0$ are trivial. 
\end{cor}
\begin{proof}
This follows from the theorem \ref{proofofconjst} and the vanishing of Voevodsky's motivic cohomology in weight $0$. 
\end{proof}
\subsection{Some applications of Morel's splitting}
Let $S$ be a Noetherian regular scheme of finite Krull dimension. We consider in this section always the rationalized stable homotopy category $SH(S)_{\Q}$. We always have a splitting $SH(S)_\Q = D_B(S) \times SH(S)_{\Q-}$, where $D_B(S)$ is the category of Beilinson's motives. $S$ is called geometrically unibranch if for any point $s \in S$ the scheme $\Spec \sO_{S,s}^{sh}$ is irreducible, where $\sO_{S,s}^{sh}$ denotes the strict Henselization of the local ring $\sO_{S,s}$. In this case we have an equivalence $D_B(S) \cong DM(S)_{\Q}$ (see \cite[Thm. 16.1.4]{CD10}), where $DM(S)_{\Q}$ denotes the category of Voevodsky's motives. Let $X \in Sch/S$ be a scheme. Following \cite[Ex. 2.4, \S 3]{MV01} we call $X$ $\A^1$-rigid, if for any smooth scheme $U \in Sm/S$ one has a bijection $\Hom_S(U,X) \cong \Hom_S(U \times_S \A^1_S, X)$. Let us warm up with the following: 
\begin{lem}\label{lemwedge}
For any natural number $n \geq 0$ the $\P^1$-spectrum $\Sigma^{\infty}_{\P^1,+}(\G_m^{\wedge n})$ is a direct summand of $\Sigma^{\infty}_{\P^1,+}(\G_m^{ \times n})$.
\end{lem}  
\begin{proof}
We prove the lemma inductively. For $n = 0$, there is nothing to prove. We have $\Sigma^{\infty}_{\P^1,+}(\G_m^{\wedge n+1}) = \Sigma^{\infty}_{\P^1,+}(\G_m^{\wedge n}) \wedge \Sigma^{\infty}_{\P^1,+}(\G_m)$ which is by induction a direct summand of the wedge product $\Sigma^{\infty}_{\P^1,+}(\G_m^{\times n}) \wedge \Sigma^{\infty}_{\P^1,+}(\G_m) = \Sigma^{\infty}_{\P^1,+}(\G_m^{\times} \wedge \G_m)$. But we have a splitting 
$$\Sigma^{\infty}_{\P^1,+}(\G_m^{\times n} \times \G_m) = \Sigma^{\infty}_{\P^1,+}(\G_m^{\times n} \wedge \G_m) \vee \Sigma^{\infty}_{\P^1,+}(\G_m^{\times n}) \vee \Sigma^{\infty}_{\P^1,+}(\G_m),$$
which finishes the proof of the lemma. 
\end{proof}
We say $S = \Spec k$ satisfies the affine stable $\A^1$-Lefschetzt's vanishing condition if for any affine $k$-scheme $X \in Sm/k$ of dimension $n$ the rational stable $\A^1$-homology sheaves $\mathbf{H}^{st\A^1}_i(X)_{\Q} = 0$ for $i > \dim (X)$.
\begin{prop}
If $k$ is a perfect field and satisfies the affine stable $\A^1$-Lefschetzt's vanishing condition, then the Beilinson-Soule's vanishing conjecture holds over $k$. 
\end{prop}
\begin{proof}
We know that if $k$ is a perfect field then Voevodsky's motivic cohomology $H\Q^{-p,-q}(*)$ is a direct summand of $\pi_{p,q}(\mathbf{1}_\Q)$. In terms of stable $\A^1$-homology under the equivalence $SH_\Q \cong D_{\A^1,\Q}$ we have 
$$\pi_{p,q}(\mathbf{1}_\Q) = \mathbf{H}^{st\A^1}_{p-q}(\G_m^{\wedge (-q)})_\Q,$$
which is a direct summand of $\mathbf{H}^{st\A^1}_{p-q}(\G_m^{\times (-q)})_\Q$ by the lemma \ref{lemwedge}. If $k$ satisfies the affine stable $\A^1$-Lefschetzt's vanishing condition, this homology sheaves will vanish for $p > 0$, which implies then the Beilinson-Soule's vanishing conjecture. 
\end{proof}
\begin{rem}{\rm
From this we see that stable $\A^1$-homotopy of $\A^1$-rigid schemes are non-trivial in general. 
}
\end{rem}
We now recall the notion of slice filtration and slice spectral sequence of Voevodsky in \cite{Voe02}. Let $S$ be a Noetherian scheme of finite Krull dimension. There is a filtration on $SH(S)$ 
$$\cdots \inj \Sigma_{\P^1}^{q+1}SH(S)^{eff} \inj \Sigma^q_{\P^1}SH(S)^{eff} \inj \cdots SH(S),$$
where $\Sigma^q_{\P^1}SH(S)^{eff}$ denotes the localizing subcategory generated by the objects $S^{2q,q}\wedge \Sigma^{\infty}_{\P^1,+}(X)$ with $X \in Sm/S$. We write $i_q: \Sigma^q_{\P^1}SH(S)^{eff}$ for the full embedding. It admits a right adjoint $r_q: SH(S) \rightarrow \Sigma^q_{\P^1}SH(S)^{eff}$ with $r_q \circ i_q = \id$ and we define $f_q = i_q \circ r_q$. For a motivic spectrum $E \in SH(S)$, its slice $s_q(E)$ is characterized by the distinguished triangle 
\begin{equation}\label{eqslice1}
f_{q+1}E \rightarrow f_q E \rightarrow s_q E \stackrel{+1}{\rightarrow}.
\end{equation} 
By taking homotopy of the distinguished triangle \ref{eqslice1}, we obtain a long exact sequence 
\begin{equation}\label{eqslice2}
\cdots \rightarrow \pi_{p,n}(f_{q+1} E) \rightarrow \pi_{p,n}(f_q E) \rightarrow \pi_{p,n}(s_q E) \rightarrow \cdots
\end{equation}
We set $D^1_{p,q,n} = \pi_{p,n}(f_q E)$ and $E^1_{p,q,n} = \pi_{p,n}(s_q E)$. Apply the standard method of exact couple to the long exact sequence \ref{eqslice2}, we obtain an exact couple 
\begin{equation}\label{eqslice3}
\xymatrix{D^1_* \ar[rr]^{(0,-1,0)} && D^1_* \ar[ld]^{(0,0,0)} \\ & E^1_* \ar[lu]^{(-1,1,0)} }
\end{equation}
The exact couple \ref{eqslice3} gives rise to the tri-graded slice spectral sequence 
\begin{equation}\label{eqsss}
E^1_{p,q,n} = \pi_{p,n}(s_q E) \Rightarrow \pi_{p,n}(E)
\end{equation}
The differential of the slice spectral sequence \ref{eqsss} 
$$d_1^{p,q,n}: \pi_{p,n}s_q E \rightarrow \pi_{p-1,n}s_{q+1} E$$
is induced by the composition 
$$s_q E \rightarrow S^1_s \wedge f_{q+1} E \rightarrow S^1_s \wedge s_{q+1}E.$$
A motivic spectrum $E \in SH(S)$ is called convergent with respect to the slice filtration, if 
$$\bigcap_{i \geq 0}f_{q+i} \pi_{p,n}f_q E = 0, \quad \forall p,q,n \in \Z,$$
where we write $f_q \pi_{p,n}E = \im (\pi_{p,n}f_q E \rightarrow \pi_{p,n}E)$. We recall now the following result obtained by O. R\"ondigs and P. A. \O stv\ae r:
\begin{thm}\cite{RO13}\label{thmRO1}
Let $k$ be a perfect field of $char(k) \neq 2$, then 
\begin{equation}\label{eqsKO}
s_q(\mathbf{KO}) = \begin{cases}S^{2q,q} \wedge H\Z \vee \bigvee_{i < \frac{q}{2}} S^{2i+q,q} \wedge H\Z/2, \quad q \equiv 0 \mod 2, \\ \bigvee_{i < \frac{q+1}{2}}S^{2i+q,q}H\Z/2, \quad q \equiv 1 \mod 2. \end{cases}
\end{equation}
and 
\begin{equation}\label{eqsKT}
s_q(\mathbf{KT}) = \bigvee_{i \in \Z}S^{2i+q,q}H\Z/2.
\end{equation}
\end{thm}
Moreover, in \cite{RO13}, they also show a conjecture of J. Hornbostel in \cite{Horn05}:
\begin{thm}\cite[Thm. 4.4]{RO13}\label{thmRO2}
Let $S$ be a Noetherian regular scheme of finite Krull dimension such that $\frac{1}{2} \in \Gamma(S,\sO_S)$. There is a distinguished triangle in $SH(S)$ 
\begin{equation}\label{eqRO1}
\mathbf{KO} \wedge \G_m \stackrel{\eta}{\rightarrow} \mathbf{KO} \stackrel{f}{\rightarrow} \mathbf{KGL} \stackrel{+1}{\rightarrow},
\end{equation}
where $\eta$ is the stable Hopf-map and $f$ denotes the forgetful map. 
\end{thm}
We show 
\begin{thm}\label{thmsssKO}
Let $k$ be a perfect field of $char(k) \neq 2$. The slice spectral sequence 
$$E^1_{p,q,n} = \pi_{p,n}(s_q \mathbf{KO}_{\Q+}) \Rightarrow \pi_{p,n}(\mathbf{KO}_{\Q+})$$
is strongly convergent and degenerates at $E^1$-terms. 
\end{thm}
\begin{proof}
By smashing $-\wedge \mathbf{1}_{\Q+}$ with the distinguished triangle \ref{eqRO1} in theorem \ref{thmRO2}, we obtain a distinguished triangle 
$$\mathbf{KO}_{\Q+} \wedge \G_m \stackrel{0}{\rightarrow} \mathbf{KO}_{\Q+} \rightarrow \mathbf{KGL}_\Q \stackrel{+1}{\rightarrow},$$
where $\eta$ becomes $0$ in $SH(k)_{\Q+}$. So we have a monomorphism $\mathbf{KO}_{\Q+} \inj \mathbf{KGL}_\Q = \mathbf{KGL}_{\Q+}$. So the strongly convergence of the slice spectral sequence for $\mathbf{KGL}_\Q$ (cf. \cite{Voe02a} and \cite{Lev08}) implies the strongly convergence of the slice spectral sequence for $\mathbf{KO}_{\Q+}$. Apply now the equation \ref{eqsKO} in theorem \ref{thmRO1} we have $s_q(\mathbf{KO}) \wedge \mathbf{1}_\Q = S^{2q,q} \wedge H\Q$ for $q \equiv 0 \mod 2$ and $0$ else. Since we may write 
$$\mathbf{1}_\Q = hocolim(\mathbb{S}^0 \stackrel{2}{\rightarrow} \mathbb{S}^0 \stackrel{3}{\rightarrow} \mathbb{S}^0 \stackrel{4}{\rightarrow} \cdots),$$
we must have $s_q(\mathbf{KO}) \wedge \mathbf{1}_\Q = s_q(\mathbf{KO} \wedge \mathbf{1}_\Q)$, because $s_q$ commutes with homotopy colimits for $\forall q \in \Z$ (see e.g. \cite[Lem. 5.5]{RO13}). But we also have 
$$s_q(\mathbf{KO} \wedge \mathbf{1}_\Q) = s_q(\mathbf{KO}_{\Q+} \vee \mathbf{KO}_{\Q-}) = s_q(\mathbf{KO}_{\Q+}) \vee s_q(\mathbf{KO}_{\Q-}) = s_q(\mathbf{KO}_{\Q+}) \vee s_q(\mathbf{KT}_\Q).$$ 
Apply now the equation \ref{eqsKT} in the theorem \ref{thmRO1}, we see $s_q(\mathbf{KO}_{\Q-}) = 0$ for all $q \in \Z$. So we obtain $s_q(\mathbf{KO}_{\Q+}) = S^{2q,q}\wedge H\Q$ for $q \equiv 0 \mod 2$ and $0$ else. So the $E^1$-terms of the slice spectral sequence for $\mathbf{KO}_{\Q+}$ are given by 
$$E^1_{p,q,n} = \begin{cases} [S^{p,n},S^{2q,q} \wedge H\Q] = H\Q^{2q-p,q-n}, \quad q \equiv 0 \mod 2, \\ 0, \quad q \equiv 1 \mod 2. \end{cases}$$
The differential $d_1^{p,q,n}$ are induced by the composition: 
$$s_q(\mathbf{KO}_{\Q+}) \rightarrow S^1_s \wedge f_{q+1}(\mathbf{KO}_{\Q+}) \rightarrow S^1_s \wedge s_{q+1}(\mathbf{KO}_{\Q+}).$$
Since $s_q(\mathbf{KO}_{\Q+}) = 0$ for $q \equiv 1 \mod 2$, so the differentials $d_1$ vanish for all $p,q,n \in \Z$. On the other hand, from the work of Naumann, Spitzweck and \O stv\ae r on motivic Landweber exactness \cite{NSO09}, we know that $H\Q$ is the Landweber exact spectrum associated with the additive formal group law over $\Q$. There is a short exact sequence \cite[Thm. 9.5]{NSO09}
\begin{multline}\label{eqLandweber1}
0 \rightarrow Ext^{1,(p-1,q)}_{MGL_{*,*}}(MGL_{*,*}H\Q,H\Q_{*,*}) \rightarrow H\Q^{p,q}H\Q \rightarrow  \\ \rightarrow  \Hom^{p,q}_{MGL_{*,*}}(MGL_{*,*}H\Q,H\Q_{*,*}) \rightarrow 0,
\end{multline}
where we denote by $MGL$ the Voevodsky's algebraic cobordism spectrum. The higher differentials $d_r$ of the slice spectral sequence are given by 
$$\pi_{p,n}(s_q \mathbf{KO}_{\Q+}) \rightarrow \pi_{p-1,n}(s_{q+r} \mathbf{KO}_{\Q+}),$$
which are induced by $s_q(\mathbf{KO}_{\Q+}) = S^{2q,q} \wedge H\Q \rightarrow S^1_s \wedge S^{2q+2r,q+r} \wedge H\Q = S^1_s \wedge s_{q+r}(\mathbf{KO}_{\Q+})$. Apply now the exact sequence \ref{eqLandweber1} we see that all the higher differentials $d_r$ vanish. Thus the slice spectral sequence degenerates at $E^1$-terms.        
\end{proof}
Consequently, we obtain the following splitting of $\mathbf{KO}_{\Q+}$, which is a result of F. Morel: 
\begin{cor}\label{corsssKO}
Let $k$ be a perfect field of $char(k) \neq 2$. There is a splitting in $SH(k)_\Q$ 
$$\mathbf{KO}_{\Q+} = \bigvee_{m \in \Z} \P^{1 \wedge 2m}_{\Q+}.$$
\end{cor}
\begin{proof}
This follows immediately from the degeneration of the slice spectral sequence in theorem \ref{thmsssKO} and the cellularity of $\mathbf{KO}$ in theorem \ref{thmKO}.
\end{proof}
\section{Twisted Chow-Witt correspondences, Chow-Wit- and Witt-motives}
We recall the intersection theory in Chow-Witt rings in \cite{Fas07} and \cite{Fas08}. Let $k$ be a field of characteristic different from $2$, $X \in Sm/k$ be a smooth $k$-scheme and $\sL$ a line bundle on $X$. The $i$-th Chow-Witt complex twisted by $\sL$ is defined to be the fiber product complex \cite[Def. 3.21]{Fas07}
\begin{equation*}
\xymatrix{C^*(X,G^i,\sL) \ar[d] \ar[r] & C^*_{FS}(X,I^i,\sL) \ar[d] \\ C^*(X,K^M_i) \ar[r] & C^*_{FS}(X,I^i/I^{i+1})}
\end{equation*}  
This means that one has a short exact sequence of complexes 
$$0 \rightarrow C^*(X,G^i,\sL) \rightarrow C^*_{FS}(X,I^i,\sL) \oplus C^*(X,K^M_i) \rightarrow C^*_{FS}(X,I^i/I^{i+1}) \rightarrow 0.$$
Here $C^*(X,K^M_i)$ is the Gersten-Milnor complex 
$$0 \rightarrow K_i^M(k(X)) \rightarrow \bigoplus_{x \in X^{(1)}}K_{i-1}^M(\kappa(x)) \rightarrow \cdots $$ 
$C^*_{FS}(X,I^i,\sL)$ is the twisted by $\sL$ complex of fundamental ideal of the Witt ring of modules of finite length in Fasel's notation. The $i$-th twisted by $\sL$ Chow-Witt group $\widetilde{\CH}^i(X,\sL)$ is defined as $H^i(X,C^*(X,G^i,\sL))$. There is a product (cf. \cite[Rem. 6.2]{Fas07})
$$\cdot : H^i(C^*(X,G^i,\sL)) \otimes H^j(C^*(X,G^j,\sL')) \rightarrow H^{i+j}(C^*(X,G^{i+j},\sL\otimes \sL')), $$
which is not commutative in general. This product defines then the intersection product $\cdot$ on the twisted Chow-Witt groups $\CH^*(X,\sL)$. For $\alpha \in \CH^i(X,\sL)$ and $\beta \in \CH^j(X,\sL')$, we define their left resp. right intersection product as
\begin{align*} 
\alpha \cdot_l \beta = \alpha \cdot \beta, \\ 
\alpha \cdot_r \beta = \beta \cdot \alpha.
\end{align*}
Recall the following proper base change result of J. Fasel: 
\begin{lem}\label{lembasechangeCW} \cite[Prop. 4.2.2]{AH11}
Given a cartesian square $\Box$ of smooth $k$-schemes 
\begin{equation*} 
\xymatrix{X' \ar[d]^g \ar[r]^q & X \ar[d]^f \\ Y' \ar[r]^p & Y} 
\end{equation*}
where $f$ is proper. Then one has $p^{!}f_* = g_*q^!$. 
\end{lem}
\begin{rem}{ \rm 
$p^! = p^*$ and $q^! = q^*$, if they are flat (see \cite[Prop. 7.4]{Fas07}). 
}
\end{rem}
\begin{cor}\label{corprojform}\cite[Cor. 4.2.3 and Rem. 4.2.4]{AH11}
Let $f: X \rightarrow Y$ be a proper morphism of smooth $k$-schemes. Then one has projection formulas for both left and right intersection products with arbitrary twists  
$$f_*(f^! \alpha \cdot_l \beta) = \alpha \cdot_l f_* \beta,$$
and 
$$f_*(f^! \alpha \cdot_r \beta) = \alpha \cdot_r f_* \beta.$$
\end{cor}
 Let $X,Y \in SmProj/k$. We define the group of Chow-Witt correspondence between $X$ and $Y$ as $\widetilde{CH}^{d_X}(X\times Y, p^{XY*}_X\omega_{X/k})$, where $p^{XY}_X$ denotes the projection $X\times Y \rightarrow X$, $\omega_{X/k}$ is the canonical line bundle on $X$ and $d_X = \dim(X)$.  
\begin{prop}\label{propCW}
Let $X,Y,Z \in SmProj/k$, where $k$ is a field of $char(k) \neq 2$. If $\alpha \in \widetilde{\CH}^{d_X}(X\times Y,p^{XY*}_X \omega_{X/k})$ and $\beta \in \widetilde{\CH}^{d_Y}(Y\times Z,p^{YZ*}_Y\omega_{Y/k})$, then their composition will land in 
$$\beta \circ \alpha \stackrel{def}{=} p^{XYZ}_{XZ*}(p^{XYZ*}_{XY}(\alpha)\cdot_l p^{XYZ*}_{YZ}(\beta)) \in \widetilde{\CH}^{d_X}(X\times Z,p^{XZ*}_X\omega_{X/k}) $$ 
Moreover, this composition is associative and satisfies the unit axiom.  
\end{prop}
\begin{proof}
The projection $p^{XYZ}_{XY}: X \times Y \times Z \rightarrow X \times Y$ induces a functorial homomorphism (cf. \cite[Thm. 3.27]{Fas07} or \cite[Chap. 10]{Fas08})
$$p^{XYZ*}_{XY}: \widetilde{\CH}^{d_X}(X \times Y,p^{XY*}_X \omega_{X/k}) \rightarrow \widetilde{\CH}^{d_X}(X\times Y \times Z, p_{XY}^{XYZ*}\circ p^{XY*}_X \omega_{X/k}).$$
The right hand-side is nothing but $\widetilde{\CH}^{d_X}(X\times Y \times Z,p^{XYZ*}_X \omega_{X/k})$. Analogously, we have a functorial homomorphism on $Y \times Z$ factor.  Consider the intersection product defined as pullback of the exterior product along the diagonal (cf. \cite[Def. 6.1]{Fas07}) 
\begin{multline*}
\cdot \, \, : \widetilde{\CH}^{d_X}(X\times Y \times Z,p^{XYZ*}_X \omega_{X/k}) \otimes \widetilde{\CH}^{d_Y}(X\times Y \times Z,p^{XYZ*}_Y \omega_{Y/k}) \rightarrow \\ \widetilde{\CH}^{d_X+d_Y}(X\times Y \times Z, p^{XYZ*}_X \omega_{X/k} \otimes_{\sO_{X\times Y \times Z}} p^{XYZ*}_Y\omega_{Y/k}). 
\end{multline*}
In general, the intersection product defined on $H^j(C^*(X,G^j,\sL))$ is not commutative (not even anti-commutative), see \cite[Rem. 6.7]{Fas07}. So we define the composition $\beta \circ \alpha$ via the left intersection product. 
$$\beta \circ \alpha \stackrel{def}{=} p^{XYZ}_{XZ*}(p^{XYZ*}_{XY}(\alpha)\cdot_l p^{XYZ*}_{YZ}(\beta)) $$
Since 
$$\omega_{X\times Y \times Z /k} = p^{XYZ*}_X\omega_{X/k} \otimes p^{XYZ*}_Y \omega_{Y/k} \otimes p^{XYZ*}_Z \omega_{Z/k},$$
we may rewrite 
\begin{multline*}
\widetilde{\CH}^{d_x+d_Y}(X\times Y \times Z,  p^{XYZ*}_X \omega_{X/k} \otimes_{\sO_{X\times Y \times Z}} p^{XYZ*}_Y\omega_{Y/k}) = \\  \widetilde{\CH}^{d_X+d_Y}(X\times Y \times Z,\omega_{X\times Y \times Z} \otimes (p^{XYZ*}_Z \omega_{Z/k})^{\vee}). 
\end{multline*}
But $(p^{XYZ*}_Z \omega_{Z/k})^{\vee} \cong p^{XYZ*}_Z (\omega_{Z/k}^{\vee})$. Apply now the pushforward $p^{XYZ}_{XZ*}$, we have  
\begin{multline*}
p^{XYZ}_{XZ*}: \widetilde{\CH}^{d_X+d_Y}(X \times Y \times Z, \omega_{X \times Y \times Z/k} \otimes p^{XYZ*}_{XZ} \circ p^{XZ*}_Z (\omega_{Z/k}^{\vee})) \rightarrow \\ \widetilde{\CH}^{d_X}(X\times Z,\omega_{X \times Z} \otimes p^{XZ*}_Z(\omega_{Z/k}^{\vee})) = \widetilde{\CH}^{d_X}(X \times Z, p^{XZ*}_X \omega_{X/k}), 
\end{multline*}
which gives us the composition $\beta \circ \alpha$ as expected. Now we check the unit axiom. Given $X,Y \in SmProj/k$ and $\alpha \in \widetilde{\CH}^{d_X}(X \times Y, p^{XY*}_X \omega_{X/k})$. Apply the definition of composition we have 
$$\Delta_Y \circ \alpha = p^{XYY}_{XY*}(p^{XYY*}_{XY}(\alpha) \cdot_l p^{XYY*}_{YY}(\Delta_Y)). $$
By projection formula for left intersection product \ref{corprojform}, the right hand side is nothing but $\alpha \cdot_l p^{XYY}_{XY*}(p^{XYY*}_{YY}(\Delta_Y))$. We write $p^{XYY}_{YY} = \id_X \times p^{YY}_Y$, so we have then 
$$p^{XYY}_{XY*}(p^{XYY*}_{YY}(\Delta_Y)) = p^{XYY}_{XY*}(1_X \times p^{YY*}_Y\Delta_Y) = 1_{X \times Y} \in \widetilde{CH}^0(X \times Y) \subset GW(k(X \times Y)),$$ 
where $GW(k(X \times Y))$ denotes the Grothendieck-Witt ring. It follows that $\Delta_Y \circ \alpha = \alpha \cdot_l 1_{X \times Y}$. Symmetrically, we have $\alpha \circ \Delta_X = 1_{X \times Y} \cdot_l \alpha$. Now the unit axiom follows from the fact that $1_{X \times Y}$ is a left and right unit (\cite[Prop. 6.8]{Fas07}). 
To check the associativity of the composition, one translates word by word from \cite[Prop. 16.1.1]{Ful98}, where one has to use \ref{lembasechangeCW}, functorialities of pullback and pushforward and the fact that the intersection product on twisted Chow-Witt groups is associative \cite[Prop. 6.6]{Fas07}. Indeed, given $X,Y,Z,W \in SmProj/k$ and $\alpha, \beta, \gamma$ correspondences from $X$ to $Y$ resp. $Y$ to $Z$ resp. $Z$ to $W$, we have then  
\begin{align*} 
\gamma \circ (\beta \circ \alpha) = p^{XZW}_{XW*}(p^{XZW*}_{XZ}(p^{XZW}_{XZ*}(p^{XYZ*}_{XY}\alpha \cdot_l p^{XYZ*}_{YZ} \beta)) \cdot_l p^{XZW*}_{ZW} \gamma)  \\ 
= p^{XZW}_{XW*}(p^{XYZW}_{XZW*}(p^{XYZW*}_{XYZ}(p^{XYZ*}_{XY}\alpha \cdot_l p^{XYZ*}_{YZ}\beta)) \cdot_l p^{XZW*}_{ZW} \gamma) \\  = p^{XZW}_{XW*}(p^{XYZW}_{XZW*}((p^{XZYW*}_{XY}\alpha \cdot_l p^{XYZW*}_{YZ}\beta)\cdot_l p^{XYZW*}_{XZW}p^{XZW*}_{ZW}\gamma)) \\ = p^{XYZW}_{XW*}((p^{XYZW*}_{XY}\alpha \cdot_l p^{XYZW*}_{YZ}\beta) \cdot_l p^{XYZW*}_{ZW} \gamma) \\ = p^{XYZW}_{XW*}(p^{XYZW*}_{XY} \alpha \cdot_l (p^{XYZW*}_{YZ}\beta \cdot_l p^{XYZW*}_{ZW}\gamma)).
\end{align*} 
Symmetrically, one can bring $(\gamma \circ \beta) \circ \alpha$ also into this form, which finishes the proof of our proposition. 
\end{proof}
From \ref{propCW} we thus may define the category of Chow-Witt correspondences $CHW(k)$. 
\begin{defn}\label{defCHW}{\rm
The category of Chow-Witt correspondences $CHW(k)$ over a field of characteristic unequal $2$ is given by $Obj(CHW(k)) = Obj(SmProj/k)$ and 
$$\Hom_{CHW(k)}(X,Y) = \widetilde{\CH}^{\dim(X)}(X\times Y, p^{XY*}_X \omega_{X/k}),$$
where composition of morphisms are defined with respect to the left intersection product. 
The category of effective Chow-Witt motives $\widetilde{CHW}^{eff}(k)$ is defined by taking the pseudo-abelian completion $CHW(k)^{\#}$. There is a canonical decomposition $\P^1 = \Spec k \oplus \mathbb{L}$. We define the category of Chow-Witt motives as $\widetilde{CHW}(k) = \widetilde{CHW}(k)^{eff}(k)[\mathbb{L}^{-1}]$. }  
\end{defn}
\begin{rem}{\rm
Instead of considering the Chow-Witt groups, one can also take the cohomology of twisted Witt-sheaf $H^{d_Y}_{Nis}(X \times_k Y, \underline{\mathbf{W}}(p_Y^{XY*}\omega_{Y/k}))$ and build then the category of geometric Witt motives, which we denote by $WM_{gm}(k)$. }
\end{rem}
Working with $\Q$-coefficient we have a splitting: 
\begin{prop}\label{propsplitCHW}
Let $k$ be a field of $char(k) \neq 2$. There is a splitting 
$$\widetilde{CHW}(k)_\Q = \underline{Chow}(k)_\Q \times WM_{gm}(k)_\Q.$$
\end{prop}
\begin{proof}
By definition, for any line bundle $\sL$ we have a pullback square of complexes  
\begin{equation*}
\xymatrix { C^*(-,K^{MW}_n,\sL) \ar[r] \ar[d] & C^*(-,K^M_n) \ar[d] \\ C^*(-,I^n,\sL) \ar[r] & C^*(-,I^n/I^{n-1}) }
\end{equation*}
Apply the Milnor's conjecture \cite{Voe03} we have $C^*(-,I^n/I^{n-1}) \cong C^*(-,K_n^M/2)$. By taking $\Q$-coefficient we obtain immediately that for any smooth projective $k$-variety $X \in SmProj/k$ and for any $\sL$ line bundle on $X$, there is a canonical splitting 
$$H^p_{Nis}(X,\underline{\mathbf{K}}^{MW}_q(\sL))_\Q = H^p_{Nis}(X,\underline{\mathbf{K}}^M_q)_\Q \oplus H^p_{Nis}(X,\underline{\mathbf{I}}^q(\sL))_\Q.$$
\end{proof}
\section{Proof of theorem \ref{mainthm}}
\subsection{Isomorphisms between the $\Hom$-groups}
Let $k$ be an infinite perfect field of characteristic unequal $2$. In this section we prove that one has a fully faithful embedding 
$$\widetilde{CHW}(k)_{\Q}^{op} \rightarrow D_{\A^1,gm}(k)_{\Q},$$ 
Remark that one has the equivalences of categories: 
$$\bold{StHo}_{\A^1,S^1}(k)_{\Q} \cong D_{\A^1}^{eff}(k)_{\Q}, \quad \bold{StHo}_{\A^1,\P^1}(k)_{\Q} \cong D_{\A^1}(k)_{\Q}.$$
By a theorem of F. Morel (cf. \cite[Thm. 4.1.8]{AH11}), one has a quasi-isomorphism of untwisted complexes
$$C^*(-,K^{MW}_i) \simeq C^*(-,K^M_i) \times_{C^*_{AH}(-,I^i/I^{i+1},\sO_X)} C^*_{AH}(-,I^i,\sO_X) ,$$
where for a line bundle $\sL$ over a smooth k-scheme $X \in Sm/k$ the complex $C^*_{AH}(-,I^i,\sL)$ in degree $j$ in Asok-Haesemeyer's notation is given by \cite[Def. 4.1.4]{AH11}
$$C^j_{AH}(X,I^i,\sL) = \bigoplus_{x \in X^{(j)}}I^i_{fl}(\sO_{X,x},\sL_x). $$ 
Hence we may redefine the twisted $i$-th Chow-Witt complex of sheaves as the fiber product complex (cf. \cite[Def. 4.1.9]{AH11})
\begin{equation*}
\xymatrix{C^*(-,K^{MW}_i, \sL) \ar[d] \ar[r] & C^*_{AH}(-,I^i,\sL) \ar[d] \\ C^*(-,K^{M}_i) \ar[r] & C^*_{AH}(-,I^i/I^{i+1})}
\end{equation*} 
The only non-zero cohomology sheaf of $C^*(-,K^{MW}_i,\sL)$ occurs in degree $0$ \cite[Thm. 4.1.10]{AH11}, and so one defines the twisted Milnor-Witt sheaf on a smooth scheme $X$ with a line bundle $\sL$ as (\cite[Def. 4.1.11]{AH11}) 
$$\bold{K}^{MW}_i(\sL) \stackrel{def}{=} H^0C^*(-,K^{MW},\sL). $$
One has obviously a canonical isomorphism from \cite[Thm. 4.1.10]{AH11} 
\begin{equation}\label{eqtwist}
H^i_{Nis}(X,\bold{K}^{MW}_i(\sL)) \cong \widetilde{\CH}^i(X,\sL).
\end{equation}
We recall in the following the result on Thom isomorphism:  
\begin{thm}(Thom isomorphism) \cite[Thm. 4.2.7]{AH11}\label{thmthomiso}
Let $X \in Sm/k$ a smooth $k$-scheme over a field $k$ of $char(k) \neq 2$ and $E/X$ be a vector bundle of rank $r$. Then one has the Thom isomorphism 
\begin{equation}\label{eqthomiso}
H^p_{Nis}(X,\bold{K}_q^{MW}(\det E)) \cong H^{p+r}_{Nis}(Th(E),\bold{K}^{MW}_{q+r}). 
\end{equation}
\end{thm}
Before we keep continuing, we will need to recall some notions and results of F. Morel. 
\begin{defn} \label{defnMor} \cite[Def. 4.7, Rem. 4.13]{Mor12}
{\rm
Let $X/k$ be an essentially smooth scheme over $k$ and $M$ be a strongly $\A^1$-invariant sheaf. Let $x \in X^{(n)}$ be a point of codimension $n$. We define 
$$C^n_{RS}(X,M) = \bigoplus_{x \in X^{(n)}} M_{(-n)}(\kappa(x), \Lambda^X_x), $$
where $\Lambda^X_x = \bigwedge_{\kappa(x)}^{top}(T_x X)$, here we write $T_x X = (\mathfrak{m}_x / \mathfrak{m}_x^2)^{\vee}$ for the tangent space at $x$ and we set 
$$M_{(-n)}(\kappa(x),\Lambda^X_x) = M_{(-n)}(\kappa(x)) \otimes_{[\kappa(x)^{\times}]} (\Lambda^X_x)^{\times}   $$ 
For any line bundle $\sL$ over $X$, we define its twisted version as 
$$C^n_{RS}(X,\sL,M_{-1}) = \bigoplus_{x \in X^{(n)}}M_{-(n+1)}(\kappa(x), \Lambda^X_x \otimes \sL_x) $$

}
\end{defn}
\begin{thm}\label{thmMor}\cite[Thm. 4. 31]{Mor12}
Let $X$ be an essentially smooth $k$-scheme, then the Rost-Schmid complex $C^*_{RS}(X,M)$ is a complex, where the boundary is defined in \cite[Defn. 4.11]{Mor12} 
\end{thm}
\begin{prop}\label{propMor}\cite[Cor. 4.43, 4.44]{Mor12}
Let $X$ be an essentially smooth $k$-scheme and $M$ be a strongly $\A^1$-invariant sheaf of abelian groups. One has a canonical isomorphism 
$$H^*(C^*_{RS}(X,M)) \cong H^*_{Nis}(X,M) \cong H^*_{Zar}(X,M). $$
There is moreover a canonical isomorphism between the Rost-Schmid complex and Gersten complex, which is compatible with pullback morphisms through smooth morphisms.  
\end{prop}
We prove firstly 
\begin{lem}\label{lemmainthm1} 
Let $X, Y \in Sm/k$ be smooth $k$-schemes, where $k$ is an infinite perfect field of $char(k) \neq 2$. Denote by $d_Y = \dim(Y)$ and $n_Y$ the rank of the vector bundle $V_Y/Y$ (see \ref{ThmVoeThom}). We write $\sE = 0_X \times V_Y/X \times Y$. For a natural number $m >> 0$ big enough, there is then a canonical isomorphism 
\begin{equation*} 
H^{2(d_Y+n_Y+m)}_{\A^1}(Th(\sE)\wedge \P^{1 \wedge m},\Q_{\A^1}(d_Y+n_Y+m)) \cong H^{d_Y + n_Y+m}_{Nis}(Th(\sE) \wedge \P^{1 \wedge m},\bold{K}^{MW}_{d_Y + n_Y+m})_{\Q}
\end{equation*}
\end{lem}
\begin{proof} 
Consider the hypercohomology spectral sequence 
$$E^{p,q}_2 = H^p_{Nis}(Th(\sE) \wedge \P^{1 \wedge m},\underline{H}^q(\Z_{\A^1}(n_Y+d_Y+m))) \Rightarrow H^{p+q}_{\A^1}(Th(\sE) \wedge \P^{1 \wedge m},\Z_{\A^1}(n_Y+d_Y+m)).$$
As $\underline{H}^q(\Z_{\A^1}(n_Y+d_Y+m)) = 0$ for $q > n_Y + d_Y+m$ (a consequence of Morel's stable $\A^1$-connectivity result \cite{Mor05}), we have a canonical homomorphism 
$$H^{2(d_Y+n_Y+m)}_{\A^1}(Th(\sE) \wedge \P^{1 \wedge m},\Z(n_Y+d_Y+m)) \rightarrow H^{d_Y+n_Y+m}(Th(\sE) \wedge \P^{1 \wedge m},\bold{K}^{MW}_{n_Y+d_Y+m}),$$
whose kernel and cokernel are built out of the groups 
$$H^{2(d_Y+n_Y+m)-i}_{Nis}(Th(\sE) \wedge \P^{1 \wedge m},\underline{H}^i(\Z_{\A^1}(n_Y+d_Y+m)))$$ and 
$$H^{2(n_Y+d_Y+m)-i+1}_{Nis}(Th(\sE)\wedge \P^{1 \wedge m},\underline{H}^i(\Z_{\A^1}(n_y+d_Y+m))),$$ 
where $i<n_Y+d_Y+m$. By purity, we have  
\begin{multline}\label{eqpur} 
H^{2(d_Y+n_Y+m)-i}(Th(p_Y^{XY*}V_Y) \wedge \P^{1 \wedge m},\underline{H}^i(\Z_{\A^1}(d_Y+n_Y+m))) \cong \\ H^{2(d_Y+n_Y+m)-i}_{X\times Y}(p_Y^{XY*}(V_Y \oplus [m]),\underline{H}^i(\Z_{\A^1}(d_Y+n_Y+m))), 
\end{multline} 
where we think $\P^{1 \wedge m}$ as the Thom space of the rank $m$ trivial bundle $[m]$ on $\Spec k$. The last group in the isomorphism \ref{eqpur} could be computed by using Rost-Schmid complexes \cite{Mor12}. Indeed, for a smooth $k$-scheme $X$ let $C^*_{RS}(X,M)$ denote the Rost-Schmid complex defined in \cite[Def. 4.11]{Mor12}, where $M$ is a strictly (strongly) $\A^1$-invariant sheaf. By \cite[Lem. 4.35]{Mor12}, one has a canonical isomorphism of complexes 
\begin{equation*}
C^{*-n_Y}_{RS}(X \times Y, \det (p_Y^* (V_Y \oplus [m])),M_{-(n_Y)}) \cong C^*_{RS, X \times Y}(p_Y^* V_Y,M), 
\end{equation*}
where we identify the vector bundle $p_Y^* (V_Y\oplus [m])$ with the normal bundle $N_{X\times Y, p_Y^* (V_Y \oplus [m])}$ and the left hand side denotes the twisted Rost-Schmid complex (see \cite[Rem. 4.13]{Mor12}). By definition we have
\begin{equation}\label{eqrs-term}
C^{j}_{RS}(X \times Y, \det (p_Y^* (V_Y\oplus [m])), M_{-(n_Y)}) = \bigoplus_{x \in (X \times Y)^{j}}M_{-(j+n_Y)}(\kappa (x), \Lambda^{X \times Y}_x \otimes \omega_{Y,p_Y(x)}), 
\end{equation}
where $\Lambda^{X \times Y}_x = \bigwedge^{top}_{\kappa (x)}(-T_{X \times Y,x})$. So by definition (see \cite[p. 139]{Mor12}), the direct summand in \ref{eqrs-term} is nothing but $M_{-(m+2n_Y)}(\kappa (x)) \otimes_{[\kappa(x)^{\times}]} \Z[(\omega_{X,p_X(x)})^{\times}]$. So apply this for $M = \underline{H}^i(\Z_{\A^1}(n_Y+d_Y+m))$, so we see that to prove the vanishing of \ref{eqpur} for $i < n_Y+d_Y+m$, it is enough to prove the vanishing of the sheaf $\underline{H}^i(\Z_{\A^1}(d_Y+n_Y+m))_{-(2d_Y+2n_Y+2m-i)}$ for $i  < n_Y+d_Y+m$. For any finitely generated separable field $L$, the evaluation of this sheaf at $L$ is the cohomology group
\begin{equation}\label{eqhomL} 
\H^i_{Nis}(\Spec L_+ \wedge \G_m^{\wedge 2(d_Y+n_Y+m)-i},\Z_{\A^1}(d_Y+n_Y+m))  
\end{equation}
By our assumption that $m >>0$ big enough, the cohomology \ref{eqhomL} is the stable cohomology 
\begin{equation}\label{eqhomL1}
H^{2(i-(n_Y+d_Y+m))}_{st\A^1}(L,\Z(i-(n_Y+d_Y+m)))
\end{equation}
We already see that the cohomology \ref{eqhomL1} vanishes for $i < n_Y+d_Y+m$ up to torsion (from Thm. \ref{proofofconjst}), so the lemma follows. 
\end{proof}
The lemma will imply the following
\begin{cor}\label{lemmainthm2}
Under the assumptions as above, one has a canonical isomorphism
\begin{equation*}
H^{2(d_Y+n_Y)}_{st\A^1}(X_+ \wedge Th(V_Y/Y),\Q_{\A^1}(d_Y+n_Y)) \cong H^{d_Y+n_Y}_{Nis}(X_+ \wedge Th(V_Y/Y),\mathbf{K}^{MW}_{d_Y+n_Y})_\Q \\ 
\end{equation*}
\end{cor}
\begin{proof}
By definition we have 
\begin{multline*}
H^{2(d_Y+n_Y)}_{st\A^1}(X_+ \wedge Th(V_Y/Y),\Q_{\A^1}(d_Y+n_Y)) \cong \\ colim_m H^{2(d_Y+n_Y+m)}_{\A^1}(X_+ \wedge Th(V_Y/Y) \wedge \P^{1 \wedge m},\Q(d_Y+n_Y+m)).
\end{multline*}
 So the proof follows immediately from the lemma \ref{lemmainthm1} above and the next proposition. 
\end{proof}
\begin{prop}\label{propmore} \cite[Prop. 3.3.5]{AH11}
For any pointed $k$-space $(\sX,x)$, there is a canonical isomorphism 
$$H^p_{Nis}(\sX,\bold{K}^{MW}_p) \stackrel{\cong}{\rightarrow} H^{p+1}_{Nis}(\sX \wedge \P^1,\bold{K}^{MW}_{p+1}) $$
\end{prop}
Now we are able to finish the proof of the isomorphism-part in the main theorem. 
\begin{prop}\label{propembedd}
Let $k$ be an infinite perfect field of $char(k) \neq 2$. Let $X, Y \in SmProj/k$ be smooth projective k-schemes. One has a canonical isomorphism 
$$\Hom_{D_{\A^1}(k)_{\Q}}(C^{st\A^1}_*(X),C^{st\A^1}_*(Y)) \cong \widetilde{\CH}^{d_Y}(X\times Y,p^{XY*}_Y\omega_{Y/k})_{\Q} = \Hom_{CHW(k)^{op}_\Q}(X,Y). $$
\end{prop}
\begin{proof} 
Denote by $0_X$ the trivial vector bundle of rank $0$ on a scheme X. Let us denote by $d_X$ resp. $d_Y$ the dimension of $X$ resp. $Y$. We have 
\begin{align*}
\Hom_{D_{\A^1}(k)_{\Q}}(C_*^{st\A^1}(X),C_*^{st\A^1}(Y)) \stackrel{(1)}{\cong} \Hom_{D_{\A^1}(k)_{\Q}}(C_*^{st\A^1}(X) \otimes C_*^{st\A^1}(Y)^{\vee},\Q) \\ \stackrel{(2)}{\cong} \Hom_{D_{\A^1}(k)_{\Q}}(C_*^{st\A^1}(X) \otimes \widetilde{C}_*^{st\A^1}(Th(-T_Y)), \Q) \\ \stackrel{(3)}{\cong} \Hom_{D_{\A^1}(k)_{\Q}}(\widetilde{C}_*^{st\A^1}(Th(0_X)) \otimes \widetilde{C}_*^{st\A^1}(Th(V_Y)),\Q_{\A^1}(d_Y + n_Y)[2(d_Y+n_Y)]) \\ \stackrel{(4)}{\cong} \Hom_{D_{\A^1}(k)_{\Q}}(\widetilde{C}_*^{st\A^1}(Th(0_X) \wedge Th(V_Y)),\Q_{\A^1}(d_Y+n_Y)[2(d_Y+n_Y)]) \\ \stackrel{(5)}{\cong} \Hom_{D_{\A^1}(k)_{\Q}}(\widetilde{C}_*^{st\A^1}(Th(0_X \times V_Y / X \times Y)), \Q_{\A^1}(d_Y + n_Y) [2(d_y+n_Y)]) \\ \stackrel{(6)}{\cong} \Hom_{D^{eff}_{\A^1}(k)}(\widetilde{C}_*^{\A^1}(Th(0_X \times V_Y/ X \times Y),\Q_{\A^1}(  d_Y+n_Y)[2(n_Y+d_Y)]) \\ \stackrel{(7)}{\cong} H^{d_Y+n_Y}_{Nis}(Th(0_X \times V_Y/X \times Y),\bold{K}^{MW}_{d_Y+n_Y})_{\Q} \stackrel{(8)}{\cong} H^{d_Y}_{Nis}(X \times Y, \bold{K}_{d_Y}^{MW}(\det (0_X \times V_Y/ X \times Y)))_{\Q} \\ \stackrel{(9)}{\cong}\widetilde{\CH}^{d_Y}(X \times Y, p^{XY*}_Y \omega_{Y/k})_{\Q} \stackrel{(10)}{=} \Hom_{CHW(k)_{\Q}}(Y,X).
\end{align*}
We explain now all these identifications. $(1)$ follows from $Hom-\otimes$ adjunction isomorphism \ref{eqdual}. The second isomorphism $(2)$ is a consequence of duality formalism in $\P^1$-stable $\A^1$-derived category (Proposition \ref{propdual}). The third isomorphism $(3)$ is a consequence of the Theorem \ref{ThmVoeThom}, which will imply an identification $\Sigma^{\infty}_{\P^1}(Th(V_Y)) \stackrel{\cong}{\rightarrow} \Sigma^{n_Y+d_Y}_{\P^1}\Sigma^{\infty}_{\P^1}(Th(-T_Y))$. (one should beware that the notation here is not a suspension spectrum, see \cite[Exam. 3.1.4]{AH11} for the discussion). The isomorphisms $(4)$ and $(5)$ follows from definitions. $(6)$ is the corollary \ref{lemmainthm2}. The identification $(7)$ is the lemma \ref{lemmainthm1}. The identification $(8)$ is the Thom isomorphism \ref{eqthomiso} in the Theorem \ref{thmthomiso}. $(9)$ follows from the simple observation that 
$$\det (0_X \times V_Y/ X\times Y) \cong \bigwedge_{\sO_{X\times Y}}p^{XY*}_Y V_Y \cong p^{XY*}_Y \bigwedge_{\sO_Y}V_Y \cong p^{XY*}_Y \omega_{Y/k},$$ 
where the last isomorphism follows from the property of determinant. $(9)$ is the identification \ref{eqtwist} between Nisnevich cohomology groups with coefficient as twisted Milnor-Witt $K$-theory sheaves and twisted Chow-Witt groups and $(10)$ is the definition of morphisms in $CHW(k)$.         
\end{proof}
\subsection{Compatibility of composition laws}
In this section, we are dealing with the compatibility of composition laws of the main theorem. We will always work with $\Q$-coefficient. We need to check that 
$$\widetilde{CHW}(k)^{op}_\Q = \underline{Chow}(k)_\Q^{op}\times WM_{gm}(k)_\Q^{op} \rightarrow D_{\A^1,gm}(k)_\Q = DM_{gm}(k)_\Q \times D_{\A^1,gm}(\underline{W}_\Q)$$
is compatible on composition laws and sends identity morphism to identity morphism, where $D_{\A^1,gm}(\underline{W}_\Q)$ denotes the thick subcategory of the homotopy category of modules over the rational $\A^1$-Eilenberg-Maclane spectrum associated to the unramified Wit sheaf $\mathbf{Ho}(H\mathbf{W}_\Q-mod)$ generated by $\Sigma^{\infty}_{\P^1,+}(X) \wedge \mathbf{1}_{\Q-}$ for $X \in SmProj/k$. First of all, $\underline{Chow}(k)^{op} \rightarrow DM_{gm}(k)$ is the Voevodsky's embedding for $k$ a perfect field. The compatibility of the canonical isomorphisms 
$$\CH^n(X) \stackrel{(1)}{\cong} H^n(X,\underline{\mathbf{K}}^M_n) \stackrel{(2)}{\cong} H^{2n}_M(X,\Z(n))$$
with the composition laws in $\underline{Chow}(k)$ follows from the work of F. D\'eglise \cite{Deg02}, where the first isomorphism $(1)$ follows in general from the work of M. Kerz \cite{Ke09}, which is called Bloch's formula and the second isomorphism $(2)$ is due to Voevodsky \cite[Lem. 4.11]{Voe03}. It remains now to check the functoriality of $WM_{gm}(k)_\Q \rightarrow D_{\A^1,gm}(\underline{\mathbf{W}}_\Q)$. We will follow the same strategy as in \cite{Deg02}. Recall that from the proof of the proposition \ref{propembedd} that we have a bunch of isomorphisms 
\begin{multline*}
\Hom_{WM_{gm}(k)_\Q}(X,Y) \stackrel{defn}{=} A^{d_Y}(C^*(X \times_k Y,W_\Q,p_Y^{XY*} \omega_{Y/k})) \stackrel{(a)}{\cong} H^{d_Y}_{Nis}(X \times_k Y, \underline{\mathbf{W}}_\Q(p_Y^{XY*} \omega_{Y/k})) \\ \stackrel{(b)}{\cong} H^{d_Y+n_Y}_{Nis}(Th(p_Y^{XY*}V_Y),\underline{\mathbf{W}}_\Q) \stackrel{(c)}{\cong} \Hom_{SH(k)_{\Q-}}(\Sigma_{\P^1,+}^{\infty}(X) \wedge Th(V_Y/Y), H\mathbf{W}_\Q \wedge S^{d_Y+n_Y}_s) \\ \stackrel{(d)}{\cong} [\Sigma^{\infty}_{\P^1,+}(X),\Sigma^{\infty}_{\P^1,+}(Y)]_{SH(k)_{\Q-}}, 
\end{multline*}     
where we write $A^p$ for the cohomology of complexes. The isomorphism $(b)$ is the Thom isomorphism, the isomorphism $(c)$ is the definition of the $\A^1$-local Eilenberg-Maclane spectrum $H\mathbf{W}$ and the isomorphism $(d)$ follows from duality and theorem \ref{proofofconjst} that $H\mathbf{W}_\Q = \mathbf{1}_{\Q-}$. All these isomorphisms are canonical. We only have to check the compatibility of composition laws for the isomorphism $(a)$, where the twisted Witt sheaf 
$$\underline{\mathbf{W}}(p_Y^{XY*} \omega_{Y/k}) \stackrel{defn}{=} A^0(C^*_{X \times Y}(-,W,p_{Y}^{XY*} \omega_{Y/k}))$$
is obviously a strictly $\A^1$-invariant sheaf. This in turn defines in a canonical way a homotopy module $\underline{\mathbf{W}}(p_Y^{XY*} \omega_{Y/k})_*$ in the heart $\Pi(SH(k))_*$ by setting 
$$\underline{\mathbf{W}}(p_Y^{XY*} \omega_{Y/k})_{-n} = \underline{\mathbf{W}}(p_Y^{XY*} \omega_{Y/k}), \forall n \in \Z.$$ We prove now 
\begin{prop}\label{proppullback}
Let $k$ be a field of $char(k) \neq 2$. Let $f: Y \rightarrow X$ be a flat morphism of smooth $k$-scheme and $\sL/X$ be any line bundle on $X$. For all $p \geq 0$ there is a commutative diagram 
\begin{equation}\label{eqpullback1} 
\xymatrix {A^p(C^*(X,\underline{\mathbf{W}},\sL)) \ar[r]^{f^*_W} \ar[d]^{\cong} & A^p(C^*(Y,\underline{\mathbf{W}},f^* \sL)) \ar[d]^{\cong} \\ H^p_{Nis}(X,\underline{\mathbf{W}}(\sL)) \ar[r]_{f^*_{\A^1}} & H^p_{Nis}(Y,\underline{\mathbf{W}}(f^*\sL))}
\end{equation} 
For $f: Y \rightarrow X$ an arbitrary morphism between smooth $k$-schemes there is a commutative diagram
\begin{equation}\label{eqpullback3} 
\xymatrix {A^p(C^*(X,\underline{\mathbf{W}},\sL)) \ar[r]^{f^!_W} \ar[d]^{\cong} & A^p(C^*(Y,\underline{\mathbf{W}},f^* \sL)) \ar[d]^{\cong} \\ H^p_{Nis}(X,\underline{\mathbf{W}}(\sL)) \ar[r]_{f^!_{\A^1}} & H^p_{Nis}(Y,\underline{\mathbf{W}}(f^*\sL))}
\end{equation} 
\end{prop}
\begin{proof}
For any morphism $f: Y \rightarrow X$ between smooth $k$-schemes we have a canonical transformation 
$$\underline{\mathbf{W}}_X(\sL) \rightarrow f_* \underline{\mathbf{W}}_Y(f^* \sL).$$
This induces a canonical morphism in derived category $\underline{\mathbf{W}}_X(\sL) \rightarrow Rf_*\underline{\mathbf{W}}_Y(f^*\sL)$. On the other hand, if $f: Y \rightarrow$ is a flat morphism, there is a morphism of complexes defined in \cite[Chap. 9-10]{Fas08}
$$f_W^*: C^*(X,W,\sL) \rightarrow C^*(Y,W,f^*\sL).$$
This gives us a morphism in the derived category $D(Ab(X_{Nis}))$ of abelian sheaves over $X_{Nis}$: 
$$C_X^*(-,W,\sL) \rightarrow f_* C_Y^*(-,W,f^* \sL) = Rf_* C^*_Y(-,W,f^*\sL).$$
By definition, this gives rise to a commutative diagram 
\begin{equation}\label{eqpullback2}
\xymatrix{\underline{\mathbf{W}}_X(\sL) \ar[d] \ar[r] & Rf_* \underline{\mathbf{W}}_Y(f^*\sL) \ar[d] \\ C_X^*(-,W,\sL) \ar[r] & Rf_* C_Y^*(-,W,f^* \sL) }
\end{equation}
Apply now the cohomology $A^p(-)$ on the diagram \ref{eqpullback2} we are done. For $f: Y \rightarrow X$ an arbitrary morphism of smooth $k$-schemes we can always factor $f$ as a closed immersion followed by a smooth (in particular, flat) morphism. The compatibility of pullbacks for flat morphisms is proved above. It remains to show for closed immersions. Let $f= i: Y \inj X$ be a closed immersion of smooth $k$-schemes. Let $N_{Y/X}$ be normal vector bundle associated to $i$. Let $\nu = i \circ p: N_{Y/X} \stackrel{p}{\rightarrow} Y \stackrel{i}{\rightarrow} X$. Let $U = D(X,Y) \setminus N_{Y/X}$ be the complement of $N_{Y/X}$ in the deformation to the normal cone space. Denote by $\pi: X \times U \rightarrow X$ the projection. Following \cite[\S 5]{Fas07}, there is a morphism of complex 
$$\sigma_YX: C^*(X,W,\sL) \stackrel{\pi^*}{\rightarrow} C^*(X \times U, W, \pi^* \sL) \stackrel{\times t}{\rightarrow} C^*(X \times U, W, \pi^* \sL) \rightarrow C^*(N_{Y/X},(i \circ p)^* \sL).$$
From this we obtain a canonical morphism in the derived category
$$\sigma_i: C^*_X(-,W,\sL) \rightarrow R\nu_* C^*_{N_{Y/X}}(-,W,p^*i^* \sL).$$
Remark that the map $C^*_Y(-,W,\sL) \rightarrow C^*_{N_{Y/X}}(-,W,p^*\sL)$ is a quasi-isomorphism, so we obtain a canonical morphism in the derived category
$$C_X^*(-,W,\sL) \rightarrow Ri_*C^*_Y(-,W,i^* \sL).$$
One has a commutative diagram 
\begin{equation*}
\xymatrix{A^0(C^*_X(-,W,\sL)) = \underline{\mathbf{W}}_X(\sL) \ar[d] \ar[rr]^{i^*} && \underline{\mathbf{W}}_Y(i^* \sL) = A^0(C^*_Y(-,W,i^*\sL)) \ar[d] \\ C^*(X,W,\sL) \ar[r]_{\sigma_YX} & C^*(N_{Y/X},p^*i^* \sL) & \ar[l]^{p^*}  C^*(Y,W,i^* \sL) }
\end{equation*}  
From this we deduce the commutativity of the following diagram 
\begin{equation*}
\xymatrix{\underline{\mathbf{W}}_X(\sL) \ar[r] \ar[d] & Ri_* \underline{\mathbf{W}}_Y(i^* \sL) \ar[d] \\ C^*_X(-,W,\sL) \ar[r] & Ri_*C^*_Y(-,W,i^*\sL) }
\end{equation*}
This finishes the proof of the proposition. 
\end{proof}
\begin{prop}\label{proppushforward}
Let $k$ be a field of $char(k) \neq 2$. Let $f: X \rightarrow Y$ be a proper morphism of smooth $k$-schemes of relative dimension $d$ and $\sL$ be any line bundle on $Y$. For all $p \geq 0$ there is a commutative square
\begin{equation}\label{eqpushforward}
\xymatrix{ A^{p}(C^*(X,W,f^*\sL \otimes \omega_{X/k})) \ar[d]^{\cong} \ar[r]^{f_*^W} & A^{p-d}(C^*(Y,W,\sL \otimes \omega_{Y/k})) \ar[d]^{\cong} \\ H^{p}_{Nis}(X,\underline{\mathbf{W}}(f^*\sL \otimes \omega_{X/k})) \ar[r]_{f^{\A^1}_*} & H^{p-d}_{Nis}(Y, \underline{\mathbf{W}}(\sL \otimes \omega_{Y/k})) }
\end{equation}
\end{prop}
\begin{proof}
If 
\begin{equation*}
\xymatrix{X' \ar[d]^g \ar[r]^v & X \ar[d]^f \\ Y' \ar[r]_u & Y}
\end{equation*}
is a catersian square of smooth schemes with $f$ proper of relative dimension $d$, then we have a commutative diagram 
\begin{equation*}
\xymatrix{H^p_{Nis}(X,\underline{\mathbf{W}}_X(f^* \sL \otimes \omega_{X/k})) \ar[d]^{v^!} \ar[r]^{f_*} & H^{p-d}_{Nis}(Y,\underline{\mathbf{W}}_Y(\sL \otimes \omega_{Y/k})) \ar[d]^{u^!} \\ H^{p-d}_{Nis}(X',\underline{\mathbf{W}}_{X'}(v^*(f^* \sL \otimes \omega_{X/k}))) \ar[r]_{g_*} & H^{p-d}_{Nis}(Y',\underline{\mathbf{W}}_{Y'}(u^*(\sL \otimes \omega_{Y/k}))) }
\end{equation*}
which follows from the proper base change theorem for Nisnevich cohomology. Now the compatibility of the pushforwards follows from the proper base change theorem for $A^p$-cohomology (see \cite[Prop. 4.2.2]{AH11}) and the proposition \ref{proppullback}.  
\end{proof}
\begin{rem}{\rm
The exterior product $\boxtimes$ defined in $\widetilde{CHW(k)}$ via \cite[Thm. 4.15]{Fas07} is compatible with the one defined in $D_{\A^1}(k)$.
}
\end{rem}
Now we check firstly the identity axiom. Let $\Delta_X \in \widetilde{\CH}^{d_X}(X \times X,p^{XX*}_X \omega_X)$ be the identity morphism $X \rightarrow X$ in $CHW(k)$. We may view $\Delta_X$ as $(\Delta_X)_*(1)$, where $(\Delta_X)_*$ is the push forward $\widetilde{\CH}^0(X) \rightarrow \widetilde{\CH}^{d_X}(X\times X,p_X^{XX*}\omega_{X/k})$. The following proposition follows immediately from the above results: 
\begin{prop}\label{propid} 
Let $X/k$ be a smooth projective variety over $k$, where $k$ is an infinite perfect field of $char(k) \neq 2$. Then one has a commutative diagram 
\begin{equation*}
\xymatrix{\widetilde{\CH}^0(X)_{\Q} \ar[d]^{\cong} \ar[r]^{\Delta_* \quad \quad \quad \quad} & H^{d_X}(X \times X,\bold{K}^{MW}_{d_X}(p_X^{XX*}\omega_X))_{\Q} \ar[d]^{\cong} \\ D_{\A^1}(k,\Q)(C_*^{st\A^1}(X),\bold{1}_k) \ar[r]_{\Delta_* } & D_{\A^1}(k,\Q)(C_*^{st\A^1}(X),C_*^{st\A^1}(X))  } 
\end{equation*}
\end{prop}
Now we turn to the composition. In the rest we consider to composition in $D_{\A^1,gm}(k)$. We show 
\begin{prop}
The composition in $D_{\A^1,gm}(k)$ given by 
\begin{multline*}
\Hom_{D_{\A^1,gm}(k)}(C_*^{st\A^1}(X),C_*^{st\A^1}(Y)) \otimes \Hom_{D_{\A^1,gm}(k)}(C_*^{st\A^1}(Y),C_*^{st\A^1}(Z)) \rightarrow \\ \Hom_{D_{\A^1,gm}(k)}(C_*^{st\A^1}(X),C_*^{st\A^1}(Z))
\end{multline*} 
satisfies the equality $\beta \circ \alpha = p_{13*}(p_{12}^* \alpha \cup p_{23}^* \beta)$. 
\end{prop}
\begin{proof} 
For the strategy of the proof, we will follow M. Levine \cite[Chap. IV]{Lev98}. By Atiyah-Spanier-Whitehead duality, we can write the composition $\beta \circ \alpha$ as 
\begin{multline*} 
\bold{1}_k \cong \bold{1}_k \otimes \bold{1}_k \stackrel{\alpha \otimes \beta}{\longrightarrow} C_*^{st\A^1}(X)^{\vee} \otimes C_*^{st\A^1}(Y) \otimes C_*^{st\A^1}(Y)^{\vee} \otimes C_*^{st\A^1}(Z) \\ \stackrel{\id \otimes \varepsilon_Y \otimes \id}{\longrightarrow} C_*^{st\A^1}(X)^{\vee} \otimes C_*^{st\A^1}(Z)
\end{multline*}
The co-diagonal is given by $\varepsilon_Y = p_{Y*}\circ \delta_Y^* \circ \boxtimes_{Y,Y}$. So we can rewrite the composition as 
\begin{multline*}\small
\bold{1}_k \cong \bold{1}_k \otimes \bold{1}_k \stackrel{\alpha \otimes \beta}{\longrightarrow}  C_*^{st\A^1}(X)^{\vee} \otimes C_*^{st\A^1}(Y) \otimes C_*^{st\A^1}(Y)^{\vee} \otimes C_*^{st\A^1}(Z) \cong \\ C_*^{st\A^1}(X)^{\vee} \otimes  C_*^{st\A^1}(Y) \otimes C_*^{st\A^1}(Th(V_Y))(-d_Y-n_Y)[-2d_Y-2n_Y] \otimes C_*^{st\A^1}(Z)  \stackrel{\id \otimes \boxtimes_{Y,Y} \otimes \id}{\longrightarrow} \\ C_*^{st\A^1}(X)^{\vee} \otimes C_*(Y_+ \wedge Th(V_Y))(-d_Y-n_Y)[-2d_Y-2n_Y] \otimes C_*^{st\A^1}(Z) \stackrel{\id \otimes \delta^*_Y \otimes \id}{\longrightarrow} \\ C_*^{st\A^1}(X)^{\vee} \otimes C_*^{st\A^1}(Th(V_Y))(-d_Y-n_Y)[-2d_Y-2n_Y] \otimes C_*^{st\A^1}(Z) \stackrel{\id \otimes p_{Y*} \otimes \id}{\longrightarrow} \\ C_*^{st \A^1}(X)^{\vee} \otimes \bold{1}_k \otimes C_*^{st\A^1}(Z)  \cong C_*^{st\A^1}(X)^{\vee} \otimes C_*^{st\A^1}(Z),
\end{multline*}
where the diagonal $\delta_Y: Th(V_Y) \rightarrow Y_+ \wedge Th(V_Y)$ is defined by using the pullback square
\begin{equation}
\xymatrix{ V_Y \ar[d] \ar[r] & V_Y \times 0 \ar[d] \\ Y \ar[r]^{\Delta_Y} & Y \times Y}
\end{equation}
On the other hand, we may rewrite the composition as 
\begin{multline*}
\bold{1}_k = \bold{1}_k \otimes \bold{1}_k \stackrel{\alpha \otimes \beta}{\longrightarrow} C_*^{st\A^1}(X)^{\vee} \otimes C_*^{st\A^1}(Y) \otimes C_*^{st\A^1}(Y)^{\vee} \otimes C_*^{st\A^1}(Z)  \cong \\ C_*^{st\A^1}(Th(V_X))(-d_X-n_X)[-2d_X-2n_X] \otimes  C_*^{st\A^1}(Y) \otimes- \\ -\otimes C_*^{st\A^1}(Th(V_Y))(-d_Y-n_Y)[-2d_Y-2n_Y] \otimes C_*^{st\A^1}(Z) \\ \stackrel{\boxtimes_{X \times Y, Y \times Z}(\alpha \otimes \beta)}{\longrightarrow}  C_*^{st\A^1}(Th(V_X) \wedge Y_+ \wedge Th(V_Y) \wedge Z_+)(-(d_X+d_Y+n_X+n_Y))[-2(d_X+d_Y+n_X+n_Y)] \\ \stackrel{p_{13*} \circ (id_X \times \delta_Y \times \id_Z)^*}{\longrightarrow} C_*^{st\A^1}(Th(V_X) \wedge Z_+)(-d_X-n_X)[-2d_X-2n_X].  
\end{multline*}
But $\boxtimes_{X \times Y, Y \times Z}(\alpha \otimes \beta) = p_{12}^*(\alpha) \cup p_{23}^*(\beta)$, so we are done.
\end{proof}
\thanks{$\bold{Acknowledegements:}$ I am very grateful to both H\'el\`ene Esnault and Marc Levine. I thank H\'el\`ene Esnault for her encouragement and her patience during the years. I thank Marc Levine heartly for reading many earlier versions of this note and for many helpful suggestions and discussions. Without his helps I am not able to finish this project. I thank Jean Fasel for many helpful comments on an earlier version of this note. I thank F. Morel and A. Asok for their answers on my questions. I thank M. Hoyois for very useful conversation.}  
\bibliographystyle{plain}
\renewcommand\refname{References}

\end{document}